\documentclass{m2an}
\usepackage[dvips]{graphicx}
\usepackage{psfrag}
\usepackage[usenames,dvipsnames]{color}

\newtheorem{theorem}{Theorem}[section]
\newtheorem{lemma}[theorem]{Lemma}
\newtheorem{proposition}[theorem]{Proposition}
\newtheorem{corollary}[theorem]{Corollary}
\newtheorem{remark}[theorem]{Remark}

\newcommand{\di}{\,\mathrm{d}}                              
\newcommand{\dV}{\di V}
\newcommand{\dS}{\di S}

\newcommand*{\conj}[1]{\overline{#1}}

\newcommand*{\N}[1]{\left\|#1\right\|}
\newcommand*{\abs}[1]{\left|#1\right|}

\DeclareMathOperator{\re}{Re} %
\DeclareMathOperator{\diam}{diam} %

\DeclareMathOperator{\spn}{span}

\newcommand{\uu}[1]{\hbox{\boldmath$#1$}}            
\newcommand{\Uu}[1]{{\mathbf{#1}}}                   

\newcommand{\IC}{\mathbb{C}}
\newcommand{\IP}{\mathbb{P}}

\newcommand{\IR}{\mathbb{R}}

\newcommand{\bx}{{\Uu x}}
\newcommand{\bd}{{\Uu d}}
\newcommand{\bn}{{\Uu n}}
\newcommand{\ba}{{\Uu a}}
\newcommand{\bb}{{\Uu b}}

\newcommand{\bnu}{{\uu\nu}}

\newcommand{\calT}{{\mathcal T}}

\newcommand{\pw}{pw}

\newcommand{\NOmega}[1]{\left\|#1\right\|_{1,k,\Omega}}
\newcommand{\Nh}[1]{\left\|#1\right\|_{1,k,{\calT_h}}}

\newcommand{\VK}{V(K)}

\begin{document}
\title{A Plane Wave Virtual Element Method for the
    Helmholtz Problem}
\author{Ilaria Perugia}\address{Faculty of Mathematics, University of Vienna, 1090 Vienna, Austria, and
Department of Mathematics, University of  Pavia, 27100 Pavia, Italy ({\tt ilaria.perugia@univie.ac.at})}
\author{Paola Pietra}\address{Istituto di Matematica Applicata e
  Tecnologie Informatiche ``Enrico Magenes'', CNR, 27100 Pavia, Italy ({\tt paola.pietra@imati.cnr.it})}
\author{Alessandro Russo}\address{University of Milano Bicocca, 20126
  Milano, Italy ({\tt alessandro.russo@unimib.it})}
%
%
\begin{abstract} 
We introduce and analyze a virtual element method (VEM) for the Helmholtz problem with
approximating spaces made of products of low order VEM functions and
plane waves. We restrict ourselves to the 2D Helmholtz equation with impedance
boundary conditions on the whole domain boundary. The main ingredients
of the plane wave VEM scheme are:
{\em i)} a low frequency space made of VEM functions, whose basis functions are not
explicitly computed in the element interiors; {\em ii)} a proper local
projection operator onto the high-frequency space, made of plane waves; {\em iii)} an approximate
stabilization term. A convergence result for the $h$-version of the method is proved, and numerical
results testing its performance on
general polygonal meshes are presented.
\end{abstract}
%
%
\maketitle
\section*{Introduction}\label{sec:intro}

The Virtual Element Method (in short, VEM) is a generalization of the
classical finite element method recently introduced in \cite{VEMbasic,Hitchhiker}.
The key features of the VEM are:
\begin{itemize}
\item
the decomposition of the computational domain can consist of arbitrary polygons in 2D
and polyhedra in 3D;
\item
the local spaces, in general, contain polynomials, as in the classical
finite elements, but include also more general functions (usually
defined through a differential operator) whose pointwise values need
not to be computed (hence the name ``Virtual'');
\item
the VEM passes the Patch Test.
\end{itemize}
The basic VEM introduced on the Poisson problem
has then been extended to highly regular~\cite{BM2014},
non conforming~\cite{ALM2arxiv014} and discontinuous approximating spaces~\cite{BMDG2014},
to general elliptic~\cite{VEMnodal} and mixed
problems~\cite{VEMbfm,VEMmixed,VEMspaces},
as well as to Stokes~\cite{ABMV2014}, plate
problems~\cite{BMCMAME2013}, linear and non linear
elasticity~\cite{BBMSINUM2013,GTP2014,BLMarxiv2015},
and fluid flows in fractured media~\cite{BBPS2014}.

In this paper, we aim at extending the VEM approach 
to an indefinite problem. More precisely,
we present and analyze a new method, based on inserting plane
wave basis functions within the VEM framework,
in order to construct a conforming, high order method for the
discretization of the Helmholtz equation. As in the
partition of unity method (PUM, see e.g.,~\cite{MEL95,BAM96,MelBab97}), we use approximation
spaces made by products of functions that constitute a
partition of unity and plane waves.

Plane wave functions are a particular case of Trefftz functions for
the Helmholtz problem, i.e., functions belonging to the kernel of the
Helmholtz operator. Other Helmholtz-Trefftz functions are, e.g., circular/spherical waves
(also denoted as Fourier-Bessel functions, or generalized harmonic
polynomials), or Hankel functions.
The idea of inserting Trefftz basis functions within the approximating
spaces in finite element discretizations of the Helmholtz problem is due to
the fact that these spaces possess better approximation properties for
Helmholtz solutions, compared to standard polynomial spaces (see,
e.g., \cite{PWapprox}), thus similar accuracy can be obtained with
less degrees of freedom. This mitigates the strong requirements
in terms of number of degrees of freedom per wavelength due to
pollution effects \cite{BAS00}.

There are in the literature several finite element methods for
the Helmholtz problem which make use of Trefftz functions.
Besides the already mentioned PUM, which is
$H^1$-conforming, other approaches use discontinuous Trefftz basis
functions and impose interelement continuity with different
strategies: 
by least square formulations (see~\cite{STO98a,MOW99});
within a discontinuous Galerkin (DG) framework (like the ultra
weak variational formulation (UWVF) \cite{CEPhd,CED98} or its
Trefftz-DG generalization
\cite{GHP09,PVersion,LocallyRef}; see also \cite{BUM07,GAB07} for the
derivation of the UWVF as a Trefftz-DG method);
by the use of Lagrange multipliers
(see, e.g., \cite{IhBa97,FHF01,FHH03,TF2006});
through weighted residual formulations (like in the variational
theory of complex rays \cite{RLS08,LR14}, or in the wave based method
\cite{DesmetPhD,Deckers2014}).

As a first construction of a plane wave-virtual element method (PW-VEM),
we focus here on the 2D Helmholtz
problem with impedance boundary conditions on the whole domain
boundary. We restrict ourselves to low order VEM functions and we take a uniform
plane wave enrichment of the approximating spaces. On quite general
polygonal meshes, our basis functions will be then products of low order VEM
functions associated to the mesh vertices, multiplied by a linear
combination of $p$ plane waves centered at the corresponding vertices.

The VEM framework is not only used for the choice of the {\it low
frequency space}, whose basis functions are not explicitly computed
in the elements interiors, but
also
for the definition of the method and the consequent approach to
its 
analysis. A crucial step
is the definition of a proper local projection operator onto a space that has to verify
two major requirements: providing good approximation properties for the
solution of the homogeneous Helmholtz problem,
and allowing to compute exactly the bilinear form whenever one of the two
entries belongs to that space
({\it consistency requirement}).
For this reason, we define the local projector
through the Helmholtz bilinear form onto the space of
discontinuous plane wave functions.

The discrete bilinear form is then split 
into two parts: one that can be computed
exactly, thanks to the projection
operator, and 
a second one with
the role of a stabilization term and that
can be approximated.

The analysis of the method is given in an abstract form and a sufficient
condition on the stabilization
term that implies the discrete G\r{a}rding inequality (which in turn implies
convergence) is provided.
%
%

The outline of the paper is as follows. We introduce the model problem
and its PW-VEM discretization in Section~\ref{sec:method}, describing
the approximating spaces and the projection operator, as well as all
the relevant matrix blocks needed for the implementation; the choice
of the stabilization operator is not specified at this stage. In
Section~\ref{sec:analysis}, we prove an abstract convergence result
of the $h$-version of the PW-VEM which, together with best approximation estimates for
the considered approximating spaces, give convergence rates of the
method, provided that continuity and G\r{a}rding inequality are
satisfied for the discrete operator.
A sufficient condition on the stabilization term that guarantees a
G\r{a}rding inequality for the discrete operator is given in
Section~\ref{sec:stab}. An explicit form of a possible stabilization
term is then provided.  Finally, we numerically test 
in Section~\ref{sec:num} the performance
of the resulting PW-VEM.

\section{The PW-VEM method}\label{sec:method}

We introduce the PW-VEM method for the Helmholtz problem. For
simplicity, we consider the two-dimensional case.

Let $\Omega\subset\IR^2$ be a bounded, convex 
polygon.
Consider the
Helmholtz boundary value problem
\begin{equation}\label{eq:helm}
  \begin{matrix}
    & -\Delta u-k^2 u = 0 \hfill&\quad\mbox{in } \Omega,\hfill\\
    &\nabla u\cdot\bnu+ i k\, u=
    g \hfill&\quad\mbox{on } \partial\Omega,
    \hfill
  \end{matrix}
\end{equation}
where the unknown $u$ is a complex-valued function,
$k>0$ is a given wave number (the corresponding wavelength is
$\lambda=2\pi/k$), $\bnu$ is the outer normal unit vector
to $\partial\Omega$, and $i$ is the imaginary unit. To simplify the
presentation, we are considering impedance boundary conditions on the
whole $\partial\Omega$, with datum $g\in L^{2}(\partial\Omega)$.

Since we are interested in the mid- and high-frequency
  regimes, we assume that $\lambda<\diam(\Omega)$ or, equivalently, $k>2\pi/\diam(\Omega)$.

The variational formulation of~\eqref{eq:helm} reads as follows: find
$u\in H^1(\Omega)$ such that
\begin{equation}\label{eq:varform}
b(u,v)=a(u,v)
+ik\int_{\partial\Omega}u\conj{v}\dS=\int_{\partial\Omega} g\conj{v}\dS\qquad
\forall v\in H^1(\Omega),
\end{equation}
where
\[
a(u,v)=\int_\Omega\nabla u\cdot\conj{\nabla v}\dV-k^2\int_\Omega
u\conj{v}\dV;
\]
see, e.g., \cite[Prop.~8.1.3]{MEL95} for the well-posedness of the
variational problem~\eqref{eq:varform}.

\subsection{Discretization spaces}\label{sec:spaces}

Let $\calT_h=\{K\}$ be a mesh of the domain $\Omega$ made of 
polygons
$K$ with mesh width $h$, i.e., $h=\max_{K\in\calT_h}h_K$, where $h_K$
is the diameter of $K$. Given $K\in\calT_h$, we denote by $\bx_K$ the
mass center of $K$, and by $\bnu_K$ the unit normal vector to $\partial K$
pointing outside $K$. Moreover, we denote
by $n_K$ the number of edges $e$ of K, and by $V_j$ and $\bx_j$, $j=1,\ldots,n_K$,
the vertices of $K$
and their coordinates, respectively.

Let $K$ be an element of $\calT_h$. First of all we choose the {\it
  low frequency space} as  
  the (local)
finite dimensional space $\VK$
defined as
\begin{equation}
\VK=\{ v \in H^1(K), v_{|\partial K}\in C^0(\partial K),\, v_{|e}\in \IP_1(e)\, \forall e \subset \partial K, \Delta v=0 \mbox{ in } K\}.
\end{equation}
Here and in the following, we denote by $\IP_1(D)$ the space of
  polynomials of degree at most one in the domain $D$.
The functions in $\VK$ are completely determined by their value in the
$n_K$ vertices. The dimension of 
$\VK$ is equal to $n_K$. We denote by
$\{\varphi_j\}_{j=1}^{n_K}$ the canonical basis functions in $\VK$ defined by
$$
\varphi_j(V_i)= \delta_{ij},\quad i,j=1,\ldots,n_K.
$$
Notice that $\IP_1(K)\subseteq V(K)$, and $\IP_1(K)=V(K)$ if
  and only if $n_K=3$. 

The basis $\{\varphi_j\}_{j=1}^{n_K}$ is a partition of 
  unity, i.e., $f(\bx)=\sum_{j=1}^{n_K}\varphi_j(\bx)=1$ for all
  $\bx\in K$. In fact,
since the functions $\varphi_j$ are harmonic in $K$, linear on each
$e\subset\partial K$, with $\varphi_j(V_i)= \delta_{ij}$,
then $\Delta f=0$
in $K$ and $f=1$ on $\partial K$, which imply $f=1$ in $K$.

Next, we introduce $p_K$ different directions
$\{\bd_\ell\}_{\ell=1}^{p_K}$, and we define the local PW-VEM space
\[
V_{p_K}(K)=\Big\{v:\
v=
\sum_{j=1}^{n_K}
\sum_{\ell=1}^{p_K}a_{j\ell}\,\varphi_j(\bx)e^{ik\bd_\ell\cdot(\bx-\bx_j)},\
a_{j\ell}\in\IC\Big\}.
\]
We also introduce the standard plane wave space (centered at $\bx_K$)
\[
V_{p_K}^\ast(K)=\Big\{v:\
v=\sum_{\ell=1}^{p_K}a_{\ell}\,e^{ik\bd_\ell\cdot(\bx-\bx_K)},\ a_{\ell}\in\IC\Big\}.
\]
Clearly, $V_{p_K}^\ast(K)\subset V_{p_K}(K)$.

Setting, for $\ell=1,\ldots,p_K$ and 
$j=1,\ldots,n_K$,
\[
\begin{split}
\pw_\ell(\bx)&=e^{ik\bd_\ell\cdot(\bx-\bx_K)},\quad
\pw_{j\ell}(\bx)=e^{ik\bd_\ell\cdot(\bx-\bx_j)},\\
\psi_r(\bx)&=\varphi_j(\bx)\,\pw_{j\ell}(\bx), \quad \mbox{ with } r=(j-1)p_K+\ell,
\end{split}
\]
we have
\[
V_{p_K}(K)=\spn\{\psi_r(\bx),\ r=1,\ldots,n_K p_K\}.
\]
Notice that $\pw_{j\ell}(\bx)=c_{j\ell}\,\pw_{\ell}(\bx)$,
  where $c_{j\ell}=e^{ik\bd_\ell\cdot(\bx_K-\bx_j)}$.

Finally, we define the global PW-VEM space
\[
V_p(\calT_h)=\big\{ v\in C^0(\conj{\Omega}):\ v|_K\in V_{p}(K)\
\forall K\in\calT_h \big\},
\]
where we have chosen $p_K=p$ and the same directions $\{\bd_\ell\}_{\ell=1}^{p}$
for all $K\in\calT_h$, which
  allows to impose continuity across interelement boundaries.

By discretizing~\eqref{eq:varform} in the spaces $V_p(\calT_h)$ one
obtain the PUM method \cite{MEL95,BAM96,MelBab97}.
The space $V(K)$ is the polygonal finite element space with
  harmonic barycentric coordinates (see, e.g., \cite{ST2004}), that can 
be considered also as VEM space of lowest order;
  see~\cite{MRS2014} for other choices of generalized barycentric coordinates. Here, we
  adopt the VEM framework and avoid quadrature on polygons, as well as
  the expression of the basis functions
  $\varphi_j$ in the element interiors.

For the validity 
of approximation estimates in plane
  wave spaces (see \cite{PWapprox,AndreaPhD} and the proof of
    Proposition~\ref{prop:inf-sup} below), we make the
  following assumptions on the meshes and on the plane wave
  directions:
\begin{itemize}
\item[{\em   i)}] there exist $\rho\in (0,1/2]$ and $0<\rho_0<\rho$  such that every $K\in\calT_h$ contains a
  ball of radius $\rho h_K$ and is star-shaped with respect to a ball of radius $\rho_0h$;
\item[{\em  ii)}] there exists an integer $m\ge 1$ such that
\[
p=2m+1;
\]
\item[{\em iii)}] the directions $\{\bd_\ell\}_{\ell=1}^{p}$ satisfy a
  minimum angle condition, i.e., there exists $0<\delta\le 1$ such
  that the minimum angle between two different directions is $\ge
  \frac{2\pi}{p}\delta$, and are such that the angle between
    two subsequent directions is $<\pi$.
\end{itemize}

\subsection{The projector $\Pi$}\label{sec:projector}

In the VEM framework, a crucial step is the choice of a local projection operator that allows to compute bilinear forms without the need of having at disposal the explicit form of the basis functions (the functions $\psi_r$ here).
The space where
projecting onto has to fulfill two major requirements: providing good approximation properties for the solution of the problem at hand, 
and allowing to compute exactly the bilinear form whenever one of the two entries belongs to that space.

Here, we choose as space where
projecting onto the space
of (discontinuous) piecewise plane waves, which we denote by
$V_p^\ast(\calT_h)$, i.e.,
\[
V_p^\ast(\calT_h)=\displaystyle{\prod_{K\in\calT_h}}V_p^\ast(K).
\]

For any $K\in\calT_h$, we define the local bilinear form
\[
a^K(u,v)=\int_K\nabla u\cdot\conj{\nabla v}\dV-k^2\int_K u\conj{v}\dV,
\]
which is a local version of the bilinear form defining the variational
problem~\eqref{eq:varform},
ignoring the boundary term.

We define the projector $\Pi: V_{p}(K)\to V_{p}^\ast(K)$ as
follows:
\begin{equation}\label{def:proj}
a^K(\Pi u,w)=a^K(u,w)\qquad \forall w\in V_{p}^\ast(K).
\end{equation}
The projector $\Pi$
is well-defined, provided that $k^2$
is not a Neumann-Laplace eigenvalue on $K$. This is guaranteed, for
instance, when condition~\eqref{eq:conditionhk} below is satisfied
(see Section~\ref{contPi} for more details).

Clearly, if $u\in V_{p}^\ast(K)$, then $\Pi u=u$.
We observe that the choice of the {\it low frequency space} $\VK$ and of
the {\it high frequency space} $V_{p}^\ast(K)$ allows us to compute 
the right hand side of (\ref{def:proj}), even if we do not know the
expression of the functions of $\VK$ in the interior of $K$. In fact, for any $u\in
V_{p}(K)$ and $w \in V_{p}^\ast(K)$, integrating by parts, we have
\[
a^K(u,w)=
\int_K \nabla u\cdot\conj{\nabla w}\dV-k^2 \int_K
u\conj{w}\dV
=\int_{\partial K} u\conj{\nabla w \cdot\bnu_K
}\dS,
\]
since $\Delta w + k^2 w =0$. The integral on $\partial K$ can be
computed because both $u$ and $\nabla w\cdot\bnu_K$ are known on $\partial
K$.

\subsection{Continuity of the operator $\Pi$}\label{contPi}

Let $\mu_2$ be the smallest strictly positive eigenvalue of the
Neumann-Laplace operator on $K$. If $K$ is convex, then $\mu_2\ge
\pi^2/h_K^2$ \cite{PW1960}.
If $K$ is star-shaped with respect to a ball, $\mu_2\ge
C_0\,\pi^2/h_K^2$,
with $0<C_0\le 1$ only depending on the shape of $K$ \cite{BP1962,LL1978}.
Therefore, assuming that
\begin{equation}\label{eq:conditionhk}
0<h_K k\le C_1 < \sqrt{C_0}\,\pi,
\end{equation}
we guarantee that $k^2<\mu_2$.

We denote by $\N{\cdot}_{0,D}$ the $L^2$-norm in the domain
  $D$, and
define the weighted norm
\[
\N{v}_{1,k,K}^2=\N{\nabla v}_{0,K}^2+k^2\N{v}_{0,K}^2\qquad
\forall v\in H^1(K).
\]

The following discrete inf-sup condition holds true.


\begin{proposition}\label{prop:inf-sup}
Provided that $h_K k$ satisfies
\begin{equation}\label{eq:condition_infsup}
h_Kk\le \alpha_0 <
\min\Big\{
\frac{\sqrt{C_0}\,\pi}{\sqrt{2}},0.5538
\Big\},
\end{equation}
there exists a positive constant $\beta=\beta(h_K k)$ 
(i.e., $\beta$ depends on
$h_K$ and $k$ through their product $h_K k$), 
which remains uniformly bounded away from zero as $h_K k\to 0$,
such that
\begin{equation}\label{eq:infsupdiscr}
\forall v\in V_p^\ast(K)\quad \exists w\in V_p^\ast(K)\quad \text{s.t.}\quad
\frac{a^K(v,w)}{\N{w}_{1,k,K}}\ge \beta(h_K k)\N{v}_{1,k,K}.
\end{equation}
\end{proposition}
\begin{proof}

Assume, with no loss of generality, that the direction
$\bd_1^\ast=(1,0)$ belongs to the set of directions
$\{\bd_\ell\}_{\ell=1}^p$.
Due to our assumptions on the directions (see assumption {\em iii)} at the end of
Section~\ref{sec:method}), there are two directions
$\bd_2^\ast,\bd_3^\ast\in \{\bd_\ell\}_{\ell=1}^p$ such that
$\bd_1^\ast,\bd_2^\ast,\bd_3^\ast$ are listed in counterclockwise
order, and the angles $\delta_1$, $\delta_2$, $\delta_3$
between $\bd_3^\ast$ and $\bd_1^\ast$, $\bd_1^\ast$, and $\bd_2^\ast$,
$\bd_2^\ast$ and $\bd_3^\ast$, respectively, satisfy 
$0<\sin\delta_1,\sin\delta_2,\sin\delta_3\le 1$.

Let $b_1 \in V_{p}^\ast(K)$ be defined as
\[
b_1(\bx)=\sum_{\ell=1}^3 \alpha_\ell e^{ik\bd_\ell^\ast\cdot(\bx-\bx_K)},
\]
with 
\[
\alpha_1=\frac{\sin\delta_3}{\sin\delta_1+\sin\delta_2+\sin\delta_3},\quad
\alpha_2=\frac{\sin\delta_1}{\sin\delta_1+\sin\delta_2+\sin\delta_3},\quad
\alpha_3=\frac{\sin\delta_2}{\sin\delta_1+\sin\delta_2+\sin\delta_3}.
\]
Since $\sum_{\ell=1}^3\alpha_\ell=1$, then $b_1(\bx_K)=1$, and since
$\sum_{\ell=1}^3\alpha_\ell\bd_\ell^\ast={\bf 0}$, then $\nabla
b_1(\bx_K)={\bf 0}$.

The following bounds hold true:
\begin{align}
\N{b_1}_{0,K} &\le \abs{K}^{1/2},\label{eq:b1L2} \\
\N{1-b_1}_{0,K} &\le \frac{1}{2} (h_K k)^2\abs{K}^{1/2},  \label{eq:b1estim}\\
\N{\nabla b_1}_{0,K} &\le h_K k^2 \abs{K}^{1/2}. \label{eq:b1H1} 
\end{align}
The bound~\eqref{eq:b1L2} follows from $\alpha_1,\alpha_2,\alpha_3>0$
and $\sum_{\ell=1}^3\alpha_\ell=1$:
\[
\begin{split}
\N{b_1}_{0,K}^2&=\int_K\abs{b_1(\bx)}^2\,dV\le
\int_K\Big(\sum_{\ell=1}^3 \alpha_\ell \Big)^2\,dV
=\abs{K}.
\end{split}
\]
The bounds~\eqref{eq:b1estim} and~\eqref{eq:b1H1} immediately follow
from Taylor's expansions of $b_1(\bx)$ and $\nabla b_1(\bx)$,
taking into account
that $b_1(\bx_K)=1$ and $\nabla b_1(\bx_K)=0$.

Fix $0\ne v\in V_{p}^\ast$ and define $w$ as the unique
element of $V_{p}^\ast$ such that
\begin{equation}\label{eq:defw}
\int_K\nabla w\cdot\conj{\nabla \xi}\,dV + k^2\int_K w\,\conj{\xi}\,dV=
a^K(v,\xi)\qquad\forall \xi\in V_{p}^\ast.
\end{equation}
Since the threshold condition~\eqref{eq:conditionhk} implies that
$k^2<\mu_2$, the right-hand side is a nonzero functional of $\xi$,
therefore $w\ne 0$.
Notice that, taking $\xi=v-w$ in~\eqref{eq:defw}, we have
$
\N{\nabla(v-w)}_{0,K}^2= k^2\N{v}_{0,K}^2-k^2\N{w}_{0,K}^2$
which implies
\begin{equation}\label{eq:aux}
\N{\nabla(v-w)}_{0,K}\le k\N{v}_{0,K}.
\end{equation}

Taking $b_1$ as test function in ~\eqref{eq:defw}, we have
\begin{eqnarray}
\int_K v\, \conj{b_1} \,dV&=&
-\int_K w\, \conj{b_1} \,dV+\frac{1}{k^2}\int_K(\nabla
v-\nabla w)\cdot \conj{\nabla b_1}\,dV\nonumber\\
&\overset{\eqref{eq:b1H1}}{\le}&
-\int_K w\, \conj{b_1} \,dV +\N{\nabla(v-w)}_{0,K} h_K
  \abs{K}^{1/2} \label{eq:integralvw}\\
&\overset{\eqref{eq:aux}}{\le}&-\int_K w\, \conj{b_1} \,dV+k\N{v}_{0,K} \,h_K
  \abs{K}^{1/2}.\nonumber
\end{eqnarray}

Moreover, by taking $\xi=w$ in~\eqref{eq:defw}, we have
\[
a^K(v,w)=\N{w}_{1,k,K}^2.
\]
In order to conclude, we need to prove that
$\N{v}_{1,k,K}\le \frac{1}{\beta(h_K k)}\N{w}_{1,k,K}$.

From the definition of $a^K(\cdot,\cdot)$ and from~\eqref{eq:defw}, 
setting $c_v=\frac{1}{\abs{K}}\int_K v \,dV$,
we have
\begin{equation}\label{eq:1normv}
\begin{split}
\N{v}_{1,k,K}^2&=a^K(v,v)+
2k^2\N{v}_{0,K}^2
=\int_K\nabla w\cdot\conj{\nabla v}\,dV + k^2\int_K
w\,\conj{v}\,dV +2k^2\N{v}_{0,K}^2\\
&\le
\N{w}_{1,k,K}\N{v}_{1,k,K}+2k^2\N{v-c_v}_{0,K}^2+2k^2\N{c_v}_{0,K}^2.
\end{split}
\end{equation}

The 
min-max principle implies 
\begin{equation}\label{eq:boundvminusc}
{\frac{\N{\nabla v}_{0,K}^2}{\N{v-c_v}_{0,K}^2}\ge \mu_2}
\ge \frac{C_0\,\pi^2}{h_K^2}
\qquad\Rightarrow\quad
2k^2 \N{v-c_v}_{0,K}^2\le \left(\frac{\sqrt{2}\,h_Kk}{\sqrt{C_0}\,\pi}\right)^2\N{\nabla v}_{0,K}^2.
\end{equation}

For the term $2k^2\N{c_v}_{0,K}^2$, 
we have
\begin{eqnarray}
%
k\N{c_v}_{0,K}&=& \frac{k}{|K|^{1/2}} \abs{\int_K v \,dV}=\frac{k}{|K|^{1/2}} \abs{
\int_K v \,\conj{b_1}\,dV+\int_K v\,(\conj{1-b_1}) \,dV}\nonumber\\
&\overset{\eqref{eq:integralvw}}{\le}&
\frac{k}{|K|^{1/2}} \abs{\int_K w\,\conj{b_1} \,dV}
+(h_Kk)k \N{v}_{0,K} 
+\frac {k}{|K|^{1/2}} \abs{\int_K v\,(\conj{1-b_1}) \,dV}\nonumber\\
&\le& 
\frac{k}{|K|^{1/2}}\N{w}_{0,K}\N{b_1}_{0,K}
+(h_Kk)k \N{v}_{0,K}  
+\frac{k}{|K|^{1/2}}\N{v}_{0,K}\N{1-b_1}_{0,K} \label{eq:boundc}\\
&\overset{\eqref{eq:b1L2},\eqref{eq:b1estim}}{\le}&
k\N{w}_{0,K}
+(h_Kk)k \N{v}_{0,K}+\frac{(h_Kk)^2}{2}k\N{v}_{0,K}\nonumber\\
&=&
k\N{w}_{0,K}+
\left[1+\frac{h_Kk}{2}\right] (h_Kk) k\N{v}_{0,K}.\nonumber
\end{eqnarray}

Inserting~\eqref{eq:boundvminusc} and~\eqref{eq:boundc}
into~\eqref{eq:1normv} gives
\begin{eqnarray*}
%
\N{v}_{1,k,K}^2&\le&\N{w}_{1,k,K}\N{v}_{1,k,K}
+\left(\frac{\sqrt{2}\,h_Kk}{\sqrt{C_0}\,\pi}\right)^2\N{\nabla
    v}_{0,K}^2
+2k^2\N{w}_{0,K}^2+2\left[1+\frac{h_Kk}{2}\right]^2 (h_Kk)^2
  k^2\N{v}_{0,K}^2\\
&&
+4\left[1+\frac{h_Kk}{2}\right](h_Kk)k\N{w}_{0,K} k\N{v}_{0,K}\\
&\le&
\left(3+4 (h_Kk)+2(h_Kk)^2\right)\N{w}_{1,k,K}\N{v}_{1,k,K}\\
&&+\max\left\{\left(\frac{\sqrt{2}\,h_Kk}{\sqrt{C_0}\,\pi}\right)^2, 2\left[1+\frac{h_Kk}{2}\right]^2 (h_Kk)^2 \right\}\N{v}_{1,k,K}^2,
\end{eqnarray*}
where in the last step we also have used $k^2\N{w}_{0,K}^2\le
  k\N{w}_{0,K} k\N{v}_{0,K}$;
thus
\[
\frac{1-\max\left\{\left(\frac{\sqrt{2}\,h_Kk}{\sqrt{C_0}\,\pi}\right)^2, 2\left[1+\frac{h_Kk}{2}\right]^2 (h_Kk)^2\right\} }{3+4 (h_Kk)+2(h_Kk)^2}\,\N{v}_{1,k,K}\le \N{w}_{1,k,K},
\]
with a positive coefficient on the left-hand side, 
provided that 
$\max\left\{\left(\frac{\sqrt{2}\,h_Kk}{\sqrt{C_0}\,\pi}\right)^2,
  2\left[1+\frac{h_Kk}{2}\right]^2 (h_Kk)^2\right\} <1$, i.e., 
provided that~\eqref{eq:condition_infsup} is
  satisfied 
(if $K$ is convex, for instance, $C_0=1$, and the
condition is $h_Kk<
0.5538$).
Then,~\eqref{eq:infsupdiscr} holds with 
\[
\beta(h_K k)=
\frac{1-\max\left\{\left(\frac{\sqrt{2}\,h_Kk}{\sqrt{C_0}\,\pi}\right)^2, 2\left[1+\frac{h_Kk}{2}\right]^2 (h_Kk)^2\right\} }{3+4 (h_Kk)+2(h_Kk)^2}.
\]
\end{proof}


\begin{remark}
In the infinite dimensional case, 
under the restriction
$0 <h_K k\le \alpha_1<\frac{\sqrt{C_0}\,\pi}{\sqrt{2}}$,
the following inf-sup condition holds true:
\[
\forall v\in H^1(K)\quad \exists w\in H^1(K)\quad \text{s.t.}\quad
\frac{a^K(v,w)}{\N{w}_{1,k,K}}\ge \beta^\ast(h_K k)\N{v}_{1,k,K},
\]
with  $\beta^\ast(h_K k)=1-\left(\frac{\sqrt{2}\,h_K
    k}{\sqrt{C_0}\,\pi}\right)^2$.
This can be proved along the same lines as in the proof of
  Proposition~\ref{prop:inf-sup},
choosing $b_1 \equiv 1$ since, in this case, it is an admissible test function.

We have carried out numerical tests, in the convex case ($C_0=1$) in order to numerically determine
the inf-sup constant in~\eqref{eq:infsupdiscr}, with only the above restriction
 on the product $h_Kk$.
The results obtained seem to
indicate that the function $\beta$ is bounded from below
by the function $\beta^\ast(h_K k)$,
and thus that
the stronger restriction~\eqref{eq:condition_infsup} on $h_Kk$ we have required
in Proposition~\ref{prop:inf-sup} might not be needed in
practice.
\end{remark}

The following result is a straightforward consequence of 
the inf-sup condition~\eqref{eq:infsupdiscr}.

\begin{proposition}\label{prop:contPi}
Under the condition~\eqref{eq:conditionhk},
the operator $\Pi$ is well-defined, and the 
following local continuity continuity property holds true:
\[
\N{\Pi u}_{1,k,K}\le \frac{1}{\beta(h_K k)}\N{u}_{1,k,K}\qquad \forall u\in V_p(K).
\]
\end{proposition}

\subsection{The matrices $D$, $B$, $G$, and $P$}\label{sec:RN}

With the same notation as~\cite{Hitchhiker}, we introduce the basic
components of our PW-VEM.

Let $D$ be the matrix of size $(n_K   p,p)$ whose $\ell$-th column
contains the coefficients of the representation of the plane wave $\pw_\ell$ 
in the $V_p(K)$
basis $\{\psi_r\}$. Simple calculations, taking into account that 
$\{\varphi_j\}_{j=1}^{n_K}$ is a partition of unity,
show that the only entries of $D$
different from zero are 
\[
D((j-1)p+\ell,\ell)=e^{-ik\bd_\ell\cdot(\bx_K-\bx_j)}, \qquad
j=1,\ldots, n_K,\ \ell=1,\ldots,p.
\]

We define $B$ as the matrix of dimension $(p,n_K   p)$ such that
\[
B(\ell,r)=a^K(\psi_r,\pw_\ell),\qquad \ell=1,\ldots,p,\ r=1,\ldots,n_K   p.
\]

As already observed in the definition of $\Pi$, since
  $\Delta\pw_\ell+k^2\pw_\ell=0$, the entries of
$B$ can be computed in terms of the traces of the shape
functions on $\partial K$ only: 
\[
\begin{split}
B(\ell,r)=a^K(\psi_r,\pw_\ell)=&
\int_K {\nabla\psi}_r\cdot\conj{\nabla \pw}_\ell \dV-k^2 \int_K
\psi_r\,\conj{\pw}_\ell\dV\\
=&\int_{\partial K}\psi_r\,\conj{\nabla\pw}_\ell \cdot\bnu_K\dS
=-ik\int_{\partial K} \bd_\ell\cdot\bnu_K\, \psi_r\,\conj{\pw}_\ell\dS,
\end{split}
\]
where we have used $\conj{\nabla\pw}_\ell=
-ik\,\bd_\ell\, \conj{\pw}_\ell$.
The integral defining $B(\ell,r)$ can be computed exactly, as
  shown in Remark~\ref{rem:exactintegrals} below.

Finally, $G$ is the matrix of size $(p,p)$ defined by
\[
G(\ell,m)=a^K(\pw_m,\pw_\ell),\qquad m,\ell=1,\ldots,p.
\]
The matrix $G$ is Hermitian and invertible, provided that
  $h_Kk$ is sufficiently small
  (see~\eqref{eq:conditionhk}).
The matrix $G$ can be expressed in terms of $B$ and $D$ as
\[
G=BD.
\]
Actually, $G$ can be computed directly, without using
  quadrature formulae (see Remark~\ref{rem:exactintegrals} below). $B$ and $D$ have to be computed in any case.
 
The matrix representation of the operator $\Pi$ is $G^{-1}B$,
thus the matrix $P$ of size $(n_K   p,n_K  p)$ defined by
\[
P=DG^{-1}B
\]
is the matrix representation of the composition of the inclusion of
$V_{p_K}^\ast(K)$ in $V_{p_K}(K)$ after $\Pi$.

\subsection{The local PW-VEM bilinear form}\label{sec:localform}

The local volume bilinear form $a^K(u,v)$ cannot be computed
  exactly on general polygons $K$ for $u,~v \in V_p(K)$. In a true PUM
  framework, once could consider as $V(K)$ the space of
    generalized barycentric coordinates and use numerical quadrature.
    On the other hand, since the basis functions in
    $V_p(K)$ are highly oscillatory, the issue of the numerical
    integration is a non trivial one. 

Here we follow the VEM paradigm and we split the local bilinear form in a part that can be computed exactly (up to machine precision),  and in a part that can be suitably approximated, provided that some stability properties are satisfied. Indeed, by using the projector $\Pi$, the local bilinear form can be written as
\[
a^K(u,v) = a^K(\Pi u,\Pi v)+a^K((I-\Pi)u,(I-\Pi)v).
\]


The term $a^K(\Pi u,\Pi v)$ can be computed  in terms of the
traces of the shape functions
on $\partial K$ only. Its matrix representation $A_{\Pi}$ is
\[
A_{\Pi} =\conj{B}^T\,\conj{G}^{-1}B.
\]

On the contrary, the computation of the term $a^K((I-\Pi)u,(I-\Pi)v)$ would require
values of all the shape functions in the interior of $K$ and it will be
approximated (see Section~\ref{sec:stab}
below).  We denote its approximation by $s^K((I-\Pi)u,(I-\Pi)v)$, 
and the local PW-VEM bilinear form can be written as
\[
a_h^K(u,v)=a^K(\Pi u,\Pi v)+s^K((I-\Pi)u,(I-\Pi)v).
\]

Since $\Pi u_p^\ast=u_p^\ast$ for all $u_p^\ast\in V^\ast_p(K)$, the
following
{\em plane wave-consistency} property holds true: 
\begin{equation}\label{eq:pwcons}
a_h^K(u_p^\ast,v)=a^K(u_p^\ast,v)\qquad \forall u_p^\ast\in
V^\ast_p(K),\ v\in V_p(K).
\end{equation}

\subsection{The global PW-VEM formulation}\label{sec:globalform}

The global bilinear form defining the PW-VEM method is given by
\begin{equation}\label{def:ah}
b_h(u,v)=a_h(u,v)+ik\int_{\partial\Omega}u\conj{v}\dS,
\end{equation}
where
\[
a_h(u,v)=\sum_{K\in\calT_h}a_h^K(u,v)=\sum_{K\in\calT_h}\left[a^K(\Pi u,\Pi v)+s^K((I-\Pi)u,(I-\Pi)v)\right],
\]
with $s^K(\cdot,\cdot)$ to be defined.
Thus, the methods reads: find $u_{hp}\in V_p(\calT_h)$ such that
\begin{equation}\label{eq:method}
b_h(u_{hp},v) =\int_{\partial\Omega}g\conj{v}\dS\qquad \forall v\in V_p(\calT_h).
\end{equation}

The boundary integral on the right-hand side of equation~\eqref{def:ah}
is computed exactly (see Remark~\ref{rem:exactintegrals} below), thus
the only integral which requires quadrature is the boundary integral
on the right-had side of~\eqref{eq:method}, containing the
inhomogeneous boundary datum. In our theoretical analysis, nevertheless, in order to avoid
complications, we assume that  also this integral is computed exactly.


\begin{remark}\label{rem:patchtest}
We point out that the {\rm plane wave-consistency}
property~\eqref{eq:pwcons} and the definition of $a_h(.,.)$ imply that
the {\rm Patch Test} is satisfied, in the following sense. On any
patch of elements, if the exact solution is a plane wave in one of the
$\bd_{\ell}$ directions that define the local spaces $V^*_p(K)$ (or a linear combination of such plane waves), then the exact solution and the approximate solution coincide.
\end{remark}

\begin{remark}\label{rem:exactintegrals}
The computation of volume and edge integrals of products of
plane waves, as well as that of edge integrals of products
of plane waves by polynomial functions, can be done exactly
(see~\cite{GAB09} and \cite[pp. 20--21]{GIT08}

In fact, if $F$ is a mesh face (edge), denoting by $\ba$ and $\bb$ the
coordinate vector of its endpoints, we have
\[
\begin{split}
\int_F
e^{ik(\bd_m-\bd_\ell)\cdot\bx}\,dS&=|F|e^{ik(\bd_m-\bd_\ell)\cdot\ba}
\int_0^1 e^{ik(\bd_m-\bd_\ell))\cdot(\bb-\ba)t}\,dt=\\
&=|F|e^{ik(\bd_m-\bd_\ell)\cdot\ba}\,\Phi_1\big(ik(\bd_m-\bd_\ell))\cdot(\bb-\ba)\big),
\end{split}
\]
where
$|F|$ denotes the length of $F$ and, for $z\in\IC$, 
\[
\Phi_1(z)=\int_0^1 e^{zt}\,dt=\begin{cases}
\displaystyle{\frac{e^z-1}{z}} &\text{if}\ z\ne 0\\
\ 1&\text{if}\ z= 0.
\end{cases}
\]
This formula also enters the computation of plane wave mass
matrices, whose entries are
integrals of the
type
\[
\int_Ke^{ik(\bd_m-\bd_\ell)\cdot\bx}\,dV.
\]
For $m=\ell$, this integral is simply $|K|$, the area of $K$. Whenever
$m\ne\ell$, using $\Delta e^{ik\bd\cdot\bx}=-k^2\bd\cdot\bd \,e^{ik\bd\cdot\bx}$, one has
\[
\int_Ke^{ik(\bd_m-\bd_\ell)\cdot\bx}\,dV=\sum_{F\in\partial K}
\frac{(\bd_m-\bd_\ell)\cdot\bn_F}{ik(\bd_m-\bd_\ell)\cdot(\bd_m-\bd_\ell)}
\int_F e^{ik(\bd_m-\bd_\ell)\cdot\bx}\,dS,
\]
where $\bn_F$ is the normal unit vector to $F$ pointing outside $K$.

For the matrix $B$,
we have seen that $B(\ell,r)=-ik\int_{\partial K}
\bd_\ell\cdot\bnu_K\, \psi_r\,\conj{\pw}_\ell\dS$. Thus, we need to compute integrals of the
type
\[
\int_F \varphi_j(\bx) e^{ik(\bd_m-\bd_\ell)\cdot\bx}\,dS.
\]
If $V_j$ is not an endpoint of $F$, then this integral is
zero. Otherwise, denoting by $\ba$ the coordinate vector of $V_j$ and
by $\bb$ the
coordinate vector of the other endpoint of $F$, we have
\[
\begin{split}
\int_F \varphi_j(\bx) e^{ik(\bd_m-\bd_\ell)\cdot\bx}\,dS&=
|F|e^{ik(\bd_m-\bd_\ell)\cdot\ba}\int_0^1
(1-t)e^{ik(\bd_m-\bd_\ell))\cdot(\bb-\ba)t}\,dt\\
&=|F|e^{ik(\bd_m-\bd_\ell)\cdot\ba}\,\Phi_2\big(ik(\bd_m-\bd_\ell))\cdot(\bb-\ba)\big),
\end{split}
\]
where 
\[
\Phi_2(z)=\int_0^1(1-t)e^{zt}\,dt=\begin{cases}
\displaystyle{\frac{e^z-z-1}{zz}} &\text{if}\ z\ne 0\\
\ 1/2&\text{if}\ z= 0.
\end{cases}
\]

Similarly, for the integral on the right-hand side of~\eqref{def:ah},
we have to compute integrals of the type
\[
\int_F\varphi_i(\bx)\varphi_j(\bx)e^{ik(\bd_m-\bd_\ell)\cdot\bx}\,dS.
\]
Whenever either $V_i$ or $V_j$ is not an endpoint of $F$, the integral
is zero. Otherwise, we distinguish two cases. If $i=j$, denoting by
$\ba$ the coordinate vector of $V_i=V_j$ and by $\bb$ the
coordinate vector of the other endpoint of $F$, we have
\[
\begin{split}
\int_F \varphi_i(\bx)\varphi_j(\bx)e^{ik(\bd_m-\bd_\ell) \cdot\bx}\,dS&=
|F|e^{ik(\bd_m-\bd_\ell)\cdot\ba}\int_0^1
(1-t)^2e^{ik(\bd_m-\bd_\ell))\cdot(\bb-\ba)t}\,dt\\
&=|F|e^{ik(\bd_m-\bd_\ell)\cdot\ba}\,\Phi_3\big(ik(\bd_m-\bd_\ell))\cdot(\bb-\ba)\big),
\end{split}
\]
where 
\[
\Phi_3(z)=\int_0^1(1-t)^2e^{zt}\,dt=\begin{cases}
\displaystyle{\frac{2(e^z-z-1)-zz}{zzz}} &\text{if}\ z\ne 0\\
\ 1/3&\text{if}\ z= 0.
\end{cases}
\]
In the second case, when $i\ne j$, denoting by
$\ba$ and $\bb$ the coordinate vector of $V_j$ and $V_i$ respectively, we have
\[
\begin{split}
\int_F \varphi_i(\bx)\varphi_j(\bx)e^{ik(\bd_m-\bd_\ell) \cdot\bx}\,dS&=
|F|e^{ik(\bd_m-\bd_\ell)\cdot\ba}\int_0^1
(1-t)t\, e^{ik(\bd_m-\bd_\ell))\cdot(\bb-\ba)t}\,dt\\
&=|F|e^{ik(\bd_m-\bd_\ell)\cdot\ba}\,\Phi_4\big(ik(\bd_m-\bd_\ell))\cdot(\bb-\ba)\big),
\end{split}
\]
where 
\[
\Phi_4(z)=\int_0^1(1-t)te^{zt}\,dt=\begin{cases}
\displaystyle{\frac{e^z(z-2)+z+2}{zzz}} &\text{if}\ z\ne 0\\
\ 1/6&\text{if}\ z= 0.
\end{cases}
\]
\end{remark}

\section{Analysis}\label{sec:analysis}

The last step in defining our discretization scheme is the choice of the stabilization term $s^K((I-\Pi)u,(I-\Pi)v)$. Before doing so, we first pose some abstract properties on the discrete problem that provide convergence
results.

\subsection{Abstract result}\label{sec:abst}
Let us introduce the $k$-dependent norm for functions in $H^1(\Omega)$:
\begin{equation*}
\NOmega{v}^2 = \N{\nabla v}_{0,\Omega}^2 + k^2 \N{v}_{0,\Omega}^2,
\end{equation*}
and the corresponding broken norm
\[
\Nh{v}^2 = \sum_{K\in\calT_h}\N{v}_{1,k,K}^2= \sum_{K\in\calT_h} (\N{\nabla v}^2_{0,K} + k^2 \N{v}_{0,K}^2),
\]
defined in the space $H^1(\calT_h)$ of broken $H^1$-functions.

The continuous bilinear form $b(\cdot,\cdot)$ satisfies the following
  continuity (see \cite[Lemma 8.1.6]{MEL95}) and G\r{a}rding inequality
\[
\begin{split}
&\abs{b(u,v)}\le C_{\text{cont}}\NOmega{u}\NOmega{v},\\
&\re[b(v,v)]+2k^2\N{v}_{0,\Omega}^2=\NOmega{v}^2
\end{split}
\]
for all $u,v\in H^1(\Omega)$.

Since for functions in $H^1(\Omega)$ the $\NOmega{\cdot}$-norm
and the $\Nh{\cdot}$-norm coincide, from here on, we will write
$\Nh{\cdot}$ for both, whenever convenient.

\begin{theorem}\label{th:abstract}
Assume that the local stabilization forms $s^K(\cdot,\cdot)$
  are chosen in such a way that the following properties hold true:
\begin{itemize}
\item {\em continuity}:
there exists $\gamma>0$ such that, for all $u,v\in
  H^1(\calT_h)$, 
\begin{equation}\label{eq:contgamma}
\abs{a_h(u,v)}\le \gamma \Nh{u}\Nh{v};
\end{equation}

\item {\em G\r{a}rding inequality for the discrete operator}: there exists
  $\alpha>0$ such that
\begin{equation}\label{eq:gardingalpha}
Re [b_h(v,v)]
+2k^2\N{v}_{0,\Omega}^2\ge \alpha \Nh{v}^2\qquad\forall
v\in{V_p(\calT_h)}.
\end{equation}
\end{itemize}

Let $u$ be the solution to problem~\eqref{eq:varform},
and let $u_{hp}$ be
solution to the PW-VEM method~\eqref{eq:method} in $V_p(\mathcal T_h)$. 
Then, provided that $h$ is small enough with respect to $k$
(see~\eqref{eq:threshold} below), the
following error estimate holds:
\[
\Nh{u-u_{hp}}\le 
C\,\frac{1+\alpha+\gamma}{\alpha}\left(
\inf_{v_I\in V_p(\mathcal T_h)}\Nh{u-v_I}
+\inf_{v_{hp}^\ast\in V_p^\ast(\calT_h)}\Nh{u-v_{hp}^\ast}
\right),
\]
with $C>0$ independent on $h$, $k$ and $p$.
Well-posedness of the PW-VEM method directly follows.
\end{theorem}

\begin{proof} 
By the triangle inequality, we have 
\begin{equation}\label{eq:triang}
\Nh{u-u_{hp}}\le \Nh{u-v_I}+\Nh{v_I-u_{hp}},
\end{equation}
for any $v_I\in V_p(\mathcal T_h)$. 
We set $\delta_{hp}=u_{hp}-v_I$, and proceed by estimating $\Nh{\delta_{hp}}$. 

The G\r{a}rding inequality~\eqref{eq:gardingalpha} gives
\begin{equation}\label{eq:alphadelta}
\alpha \Nh{\delta_{hp}}^2\leq Re
[b_h(\delta_{hp},\delta_{hp})]+2k^2\N{\delta_{hp}}_{0,\Omega}^2\,=:
\, \,I\,+\, II.
\end{equation}

The term $I$ can be treated as in \cite[Th. 3.1]{VEMbasic}.
Using the definition of $b_h(\cdot,\cdot)$ and the discrete
equation~\eqref{eq:method}, 
for any $v_{hp}^\ast$ in $V_p^\ast(\calT_h)$, we get
\begin{equation}\label{eq:estI1}
\begin{split}
&b_h(\delta_{hp},\delta_{hp})=b_h(u_{hp},\delta_{hp})-
a_h(v_I,\delta_{hp}) -\,ik\int_{\partial\Omega}v_I\conj{\delta}_{hp}\dS\\
&\quad =\int_{\partial\Omega}g\conj{\delta}_{hp}\dS -\sum_{K\in\calT_h}
a_h^K(v_I,\delta_{hp}) -\,ik\int_{\partial\Omega}v_I\conj{\delta}_{hp}\dS
\quad\mbox{(insert $v_{hp}^\ast$ and use {\em pw-consist.~\eqref{eq:pwcons}})}\\
&\quad =\int_{\partial\Omega}g\conj{\delta}_{hp}\dS - \sum_{K\in\calT_h} a_h^K(v_I-v_{hp}^\ast,\delta_{hp})
- \sum_{K\in\calT_h}
a^K(v_{hp}^\ast,\delta_{hp})-ik\int_{\partial\Omega}v_I\conj{\delta}_{hp}\dS
\quad \mbox{(use \eqref{eq:varform})}\\
&\quad =a(u,\delta_{hp})+ik\int_{\partial\Omega}u\,\conj{\delta}_{hp}\dS
- a_h(v_I-v_{hp}^\ast,\delta_{hp})
- \sum_{K\in\calT_h}
a^K(v_{hp}^\ast,\delta_{hp})-\,ik\int_{\partial\Omega}v_I\conj{\delta}_{hp}\dS\\
&\quad =\sum_{K\in\calT_h} a^K(u-v_{hp}^\ast,\delta_{hp})
+a_h(v_{hp}^\ast-v_I,\delta_{hp}) +ik\int_{\partial\Omega}(u-v_I)\conj{\delta}_{hp}\dS.
\end{split}
\end{equation}

In order to bound the last term, we observe that the trace inequality states that there exists a constant
$C>0$ only depending on the shape of $\Omega$ such that, for all $v\in H^1(\Omega)$,
\[
\begin{split}
k^{1/2}\N{v}_{0,\partial\Omega}&\le C
k^{1/2}\left(\diam(\Omega)^{-1/2}\N{v}_{0,\Omega}+
  \N{v}_{0,\Omega}^{1/2}\N{\nabla v}_{0,\Omega}^{1/2}\right)\\
&\le
C\left[k^{-1/2}\diam(\Omega)^{-1/2}\N{v}_{1,k,\Omega}+\left(\frac{k}{2}\N{v}_{0,\Omega}+\frac{1}{2}\N{\nabla
      v}_{0,\Omega}\right)\right]\\
&\le C\left(\diam(\Omega)^{-1/2} k^{-1/2}\N{v}_{1,k,\Omega}+\N{v}_{1,k,\Omega}\right).
\end{split}
\] 
Due to the high-frequency assumption, it holds
$k^{-1}\diam(\Omega)^{-1}<1/2\pi$, and we conclude that there exists a constant
$C>0$ only depending on $\Omega$ such that, for all $v\in H^1(\Omega)$,
\[
k^{1/2}\N{v}_{0,\partial\Omega}\le C\N{v}_{1,k,\Omega}.
\]

Thus, since both $(u-v_I)$ and $\delta_{hp}$ belong to $H^1(\Omega)$,  we
have
\begin{equation}\label{eq:boundest1}
\begin{split}
\big|ik\int_{\partial\Omega}(u-v_I)\conj{\delta}_{hp}\dS\big|&\le
k\N{u-v_I}_{0,\partial\Omega}\N{\delta_{hp}}_{0,\partial\Omega}\\
&\le C \NOmega{u-v_I}\NOmega{\delta_{hp}},
\end{split}
\end{equation}
with a constant $C>0$ only depending on $\Omega$.

From~\eqref{eq:estI1} and~\eqref{eq:boundest1}, using the continuity
of the bilinear form and the continuity
property~\eqref{eq:contgamma}, we
can conclude with the following estimate of the term $I$:
\[
\begin{split}
I&
=\re [b_h(\delta_{hp},\delta_{hp})]
\le C_I(1+\gamma)(\Nh{u-v_{hp}^\ast}+\Nh{u-v_I})\Nh{\delta_{hp}}.
\end{split}
\]

For the term $II$, we have
\[
\begin{split}
II=2k^2\N{\delta_{hp}}_{0,\Omega}^2&\le
2k^2\N{\delta_{hp}}_{0,\Omega}(\N{u-u_{hp}}_{0,\Omega}+\N{u-v_I}_{0,\Omega})\\
&\le 2k\Nh{\delta_{hp}}(\N{u-u_{hp}}_{0,\Omega}+\N{u-v_I}_{0,\Omega}),
\end{split}
\]
due to $k \N{\delta_{hp}}_{0,\Omega}\le \Nh{\delta_{hp}}$.
Inserting the bounds for $I$ and $II$ into~\eqref{eq:alphadelta} gives
\begin{equation}\label{eq:bounddelta}
\begin{split}
\alpha\Nh{\delta_{hp}}&\le C_I(1+\gamma)(\Nh{u-v_{hp}^\ast}+\Nh{u-v_I})+2k
(\N{u-u_{hp}}_{0,\Omega}+\N{u-v_I}_{0,\Omega})\\
&\le C_{II}(1+\gamma) (\Nh{u-v_{hp}^\ast}+\Nh{u-v_I})+2k \N{u-u_{hp}}_{0,\Omega},
\end{split}
\end{equation}
where we have used again $k \N{u-v_I}_{0,\Omega}\le \Nh{u-v_I}$.

The term $\N{u-u_{hp}}_{0,\Omega}$ can be estimated by using a duality argument.
Let us denote by $\psi$
the solution of the dual problem
\begin{equation}\label{eq:dual}
b(v,\psi)=
\int_{\Omega}
v\,\conj{(u-u_{hp})}\dV\qquad \forall v\in H^1(\Omega).
\end{equation}
Since $\Omega$ is assumed to be convex,
$\psi$ belongs to $H^2(\Omega)$ and satisfies

\begin{equation}\label{apriori_dual}
\begin{split}
&\Nh{\psi}\leq C\N{u-u_{hp}}_{0,\Omega},\\
&\abs{\psi}_{2,\Omega} \leq C(1+k)\N{u-u_{hp}}_{0,\Omega};
\end{split}
\end{equation}
see~\cite[Prop. 8.1.4]{MEL95}.
%

In correspondence to $\psi\in H^2(\Omega)$, there exist
$\psi_{hp}^\ast\in V^\ast_{p}(\calT_h)$ and $\psi_I\in V_p(\calT_h)$ and 
$\psi_I\in
V_p(\calT_h)$ such that
\begin{equation}\label{eq:approxdual}
\begin{split}
\Nh{\psi-\psi_{hp}^\ast}&\le C(1+hk)\,h\N{\psi}_{2,k,\Omega}, \\
\Nh{\psi-\psi_I}&\le C(1+hk) \,h\N{\psi}_{2,k,\Omega},
\end{split}
\end{equation}
with $C>0$ independent of $h$, $k$ and $\psi$.
The first bound follows as in~\cite[Prop. 3.12 and 3.13]{GHP09}.
The second one follows from combining \cite[Th. 2.1]{BAM96} with the local
approximation estimates in plane wave spaces.

Using (\ref{eq:dual}) with $v=u-u_{hp}$, and inserting $\psi_I$, we have
\begin{equation*}
\begin{split}
\N{u-u_{hp}}_{0,\Omega}^2=\,& b(u-u_{hp},\psi)=b(u-u_{hp},\psi-\psi_I)+b(u-u_{hp},\psi_I)\\
&\hspace{-0.75truecm}\overset{\eqref{eq:varform},\eqref{eq:method}}{=}\, b(u-u_{hp},\psi-\psi_I)+\int_{\partial\Omega}g\conj{\psi}_I\dS-b(u_{hp},\psi_I)\\
& +b_h(u_{hp},\psi_I)-\int_{\partial\Omega}g\conj{\psi}_I\dS\\
=\, &b(u-u_{hp},\psi-\psi_I)-a(u_{hp},\psi_I)-ik\int_{\partial\Omega}u_{hp}\conj{\psi}_I\dS\\
& +a_h(u_{hp},\psi_I) +ik\int_{\partial\Omega}u_{hp}\conj{\psi}_I\dS\\
=\,& b(u-u_{hp},\psi-\psi_I) +(a_h(u_{hp},\psi_I)-a(u_{hp},\psi_I))
=:III+IV.
\end{split}
\end{equation*}

The term $III$ is bounded using the continuity of the continuous form $b(\cdot,\cdot)$, 
the second estimate 
in~\eqref{eq:approxdual}
and the regularity bounds (\ref{apriori_dual}):
\begin{eqnarray*}
III &=&b(u-u_{hp},\psi-\psi_I) \leq C \Nh{u-u_{hp}} \NOmega{\psi-\psi_I}\\
& \leq& C_{III}\, (1+hk)\,h\,(1+k)\,
\Nh{u-u_{hp}}\N{u-u_{hp}}_{0,\Omega}.
\end{eqnarray*}

For the term $IV$, 
using the {\it pw-consistency} property~\eqref{eq:pwcons} 
and the continuity of the discrete forms,
 we obtain
\begin{eqnarray}\label{IV}
IV&=&a_h(u_{hp},\psi_I)-a(u_{hp},\psi_I)=\sum_{K\in\calT_h} (a_h^K(u_{hp},\psi_I)-a^K(u_{hp},\psi_I))
\nonumber\\
&=&\sum_{K\in\calT_h} (a_h^K(u_{hp}-v_{hp}^\ast,\psi_I)-a^K(u_{hp}-v_{hp}^\ast,\psi_I))\\
&=&\sum_{K\in\calT_h} (a_h^K(u_{hp}-v_{hp}^\ast,\psi_I-\psi_{hp}^\ast)-a^K(u_{hp}-v_{hp}^\ast,\psi_I
-\psi_{hp}^\ast)) \nonumber\\
&\le& 
{(1+\gamma)}\Nh{u_{hp}-v_{hp}^\ast}\Nh{\psi_I-\psi_{hp}^\ast}. \nonumber
\end{eqnarray}
Next, using the bounds in~\eqref{eq:approxdual}
and (\ref{apriori_dual}), we have
\begin{eqnarray*}
\Nh{\psi_I-\psi_{hp}^\ast} &\leq& \Nh{\psi_I- \psi}+\Nh{\psi-\psi_{hp}^\ast}\\
&\leq&C\, (1+hk) \,h\, (1+k)\,\N{u-u_{hp}}_{0,\Omega},
\end{eqnarray*}
that we insert in (\ref{IV}) getting
\begin{equation*}
IV \leq C_{IV}\, {(1+\gamma)\,} (1+hk) \,h\, (1+k)\,( \Nh{u-u_{hp}}+\Nh{u-v_{hp}^\ast})  \N{u-u_{hp}}_{0,\Omega}.
\end{equation*}

Therefore, we conclude with the following bound for $\N{u-u_{hp}}_{0,\Omega}$:
\begin{equation}\label{eq:L2-error}
\N{u-u_{hp}}_{0,\Omega}\leq (C_{III}+C_{IV})\,{(1+\gamma)\,} (1+hk) \,h\, (1+k)\,( \Nh{u-u_{hp}}+\Nh{u-v_{hp}^\ast}),
\end{equation}
which, inserted into~\eqref{eq:bounddelta}, gives
\[
\begin{split}
\alpha\Nh{\delta_{hp}}\le&\, C_{II}\,(1+\gamma)(\Nh{u-v_{hp}^\ast}+\Nh{u-v_I})\\
&+C_{\rm dual}\,{(1+\gamma)\,} k\,
(1+hk) \,h\, (1+k)\,(\Nh{u-u_{hp}}+\Nh{u-v_{hp}^\ast}),
\end{split}
\]
($C_{\rm dual}=2(C_{III}+C_{IV})$)
and thus, owing to~\eqref{eq:triang},
\[
\begin{split}
\alpha\Nh{u-u_{hp}}\le &\, C_{II}\,(1+\alpha+\gamma)(\Nh{u-v_{hp}^\ast}+\Nh{u-v_I})\\
&+C_{\rm dual}\,{(1+\gamma)\,} k\,
(1+hk) \,h\, (1+k)\,(\Nh{u-u_{hp}}+\Nh{u-v_{hp}^\ast}).
\end{split}
\]
Under the assumption
\begin{equation}\label{eq:threshold}
C_{\rm dual}\,{(1+\gamma)\,}(1+hk) \,hk\, (1+k)\le \frac{\alpha}{2},
\end{equation}
we can take the $\Nh{u-u_h}$ term to
the left-hand side and obtain
\[
\Nh{u-u_{hp}}\le
C\,\frac{1+\alpha+\gamma}{\alpha}\,(\Nh{u-v_{hp}^\ast}+\Nh{u-v_I}),
\]
with $C=2C_{II}+1$,
which completes the proof.
\end{proof}

\begin{remark}
From~\eqref{eq:threshold}, it is clear that the threshold condition
on $h$ required in Theorem~\ref{th:abstract} is that $(1+hk) \,hk\, (1+k)$ be sufficiently small,
which, in the relevant case of large $k$, is equivalent to requiring that
$hk^2$ be sufficiently small. 
This reflects the {\em pollution
effect}  of the $h$-version of the PW-VEM \cite{BAS00}. 
In fact, while a condition
on $hk$ is enough for the convergence of the best approximation
(see Proposition~\ref{prop:bestappr} below), a stronger condition (namely, on $hk^2$)
is required for the convergence of the method.
\end{remark}


The abstract convergence result of Theorem~\ref{th:abstract}, combined
with best approximation estimates of Helmholtz solutions within 
$V_p(\calT_h)$ and $V_p^{\ast}(\calT_h)$, gives convergence rates.

In order to state the following results,
we define the weighted norm 
\[
\N{u}_{s,k,\Omega}^2=\sum_{j=0}^s k^{2(s-j)}\abs{u}_{j,\Omega}^2,
\]
where $\abs{\cdot}_{j,\Omega}$ denotes the standard seminorm in
$H^j(\Omega)$.

\begin{proposition}\label{prop:bestappr}
Let $u$ be a function in $H^{\ell+1}(\Omega)$, $\ell\ge 0$, such that $\Delta u +k^2
u=0$ in $\Omega$. Then there exist $u_{hp}^\ast\in
V_p^{\ast}(\calT_h)$ and $u_I\in V_p(\calT_h)$, with $p=2m+1$, such
that 
\[
\begin{split}
\Nh{u-u_{hp}^\ast} &\le C\,\eta(hk)\,
h^{\min\{m,\ell\}}\N{u}_{\min\{m,\ell\}+1,k,\Omega},\\
\Nh{u-u_I} &\le C\,\eta(hk)\,
h^{\min\{m,\ell\}}\N{u}_{\min\{m,\ell\}+1,k,\Omega},
\end{split}
\]
with $C>0$ independent of $h$, $k$ and $u$ and
\begin{equation*}
\eta(hk)=
\left(1+(hk)^{m+9}\right)e^{\left(\frac{7}{4}-\frac{3}{4}\rho\right)hk}.
\end{equation*}
\end{proposition}

\begin{proof}
The first bound follows from the local approximation estimates in plane wave
spaces of \cite[Th. 3.2.2]{HMP09b}
(see also \cite[p. 831]{PWapprox}, \cite{AndreaPhD}).
The second bound can be obtained from
the local approximation estimates in plane wave spaces by
\cite[Th. 2.1]{BAM96}.
\end{proof}

Theorem~\ref{th:abstract} and Proposition~\ref{prop:bestappr}
immediately give the following convergence result.

\begin{corollary}\label{cor:rates}
Under the assumptions of Theorem~\ref{th:abstract}, if the solution 
$u$ to problem~\eqref{eq:varform} belongs to $H^{\ell+1}(\Omega)$,
$\ell\ge 1$, then, provided that the threshold
condition~\eqref{eq:threshold}
is satisfied, the following error estimate holds:
\[
\Nh{u-u_{hp}}\le C\,
\frac{1+\alpha+\gamma}{\alpha}
\,\eta(hk)\,
                 h^{\min\{m,\ell\}}\N{u}_{\min\{m,\ell\}+1,k,\Omega},
\]
with $C>0$ independent on $h$, $k$ and $p=2m+1$, and $\eta(hk)$ as in Proposition~\ref{prop:bestappr}.
\end{corollary}



\section{Stabilization term}\label{sec:stab}
Let us give a sufficient condition 
on the stabilization term $s^K((I-\Pi)u,(I-\Pi)v))$ in order to guarantee the G\r{a}rding inequality for
the discrete operator. To this aim, we first state the
following lemma.

\begin{lemma}\label{lemma:doubleprod}
For $u \in V_p(K)$, we have
\begin{equation*}
k^2 \int_K (u-\Pi u) \,\conj{ w} \dV=
\int_K \nabla(u -\Pi u) \cdot  \conj{
\nabla w} \dV
\qquad \forall w \in V^\ast_{p}(K).
\end{equation*}
\end{lemma}
\begin{proof}
The explicit form of the bilinear form that defines the projector $\Pi$ (see \eqref{def:proj}) gives, for all $w \in V^\ast_{p}(K)$,
\begin{equation*}
\int_K \nabla \Pi u \cdot \conj{\nabla w} \dV -
k^2 \int_K \Pi u \,\conj{w} \dV = \int_K \nabla u \cdot \conj{\nabla w} \dV
 - k^2 \int_K u \,\overline w \dV ,
\end{equation*}
that implies the assertion.
\end{proof}

\begin{proposition}\label{prop:garding}
If the stabilization form satisfies the following condition: there
exists $\alpha_s>0$ such that, for all $K\in\calT_h$ and $v\in V_p(K)$,
\begin{equation}\label{eq:stabil_hyp}
 s^K((I-\Pi)v,(I-\Pi)v) \geq \alpha_s \N{\nabla(I-\Pi)v}_{0,K}^2,
\end{equation}
then the G\r{a}rding inequality for the discrete operator holds true:
\begin{equation*}
\re[b_h(v,v)]+2 k^2 \N{v}_{0,\Omega}^2 \geq \min\{\alpha_s,1\}\Nh{v}^2 \quad \forall
v \in V_p(\mathcal T_h).
\end{equation*}
\end{proposition}

\begin{proof}
We will make use of the trivial identity
\begin{equation}\label{norm_squared}
\N{v_1+v_2}_{0,K}^2=\N{v_1}_{0,K}^2+\N{v_2}_{0,K}^2+2 \re \Big[\int_K v_1\conj{v}_2\dV\Big].
\end{equation}

Let us define  $\alpha=\min\{\alpha_s,1\}$. 
Due to 
$\alpha\le 1$,
definition (\ref{def:ah}) of $b_h(\cdot,\cdot)$ and 
identity~\eqref{norm_squared}
with $v_1=\Pi v$ and $v_2=(I-\Pi)v$,
we have
\begin{equation*}
\begin{split}
&\re[b_h(v,v)]+2 k^2 \N{v}_{0,\Omega}^2 = 
a_h(v,v)+2 k^2 \N{v}_{0,\Omega}^2\\
&\quad\geq \sum_{K\in\calT_h}  \left[\alpha  a^K(\Pi v, \Pi v) + 2 \alpha  k^2\N{\Pi v}^2_{0,K}\right] \\
&\qquad+ \sum_{K\in\calT_h}
\left[s^K((I-\Pi)v,(I-\Pi)v))
+ 2 \alpha k^2\N{(I-\Pi)v}^2_{0,K}\right] \\
&\qquad+ \sum_{K\in\calT_h} 4 \alpha k^2 \re \Big[\int_K \Pi v\,
\conj{(I-\Pi)v} \dV \Big]\\
&\quad\geq\alpha \sum_{K\in\calT_h} \left[\N{\nabla\Pi
      v}_{0,K}^2+k^2\N{\Pi v}^2_{0,K}\right]
+\sum_{K\in\calT_h}
  \left[\alpha_s\N{\nabla(I-\Pi)v}_{0,\Omega}^2 +2\alpha
    k^2\N{(I-\Pi)v}^2_{0,K}\right]\\
&\qquad+\sum_{K\in\calT_h} 
2 \alpha Re \Big[\int_K \nabla\Pi v\cdot
\conj{\nabla(I-\Pi)v} \dV\Big]
+\sum_{K\in\calT_h}  2 \alpha k^2 Re \Big[\int_K \Pi v\,
\conj{(I-\Pi)v} \dV\Big],
%
\end{split}
\end{equation*}
where in the second inequality we have used~\eqref{eq:stabil_hyp} and
Lemma~\ref{lemma:doubleprod}. 
By using~\eqref{norm_squared} with $v_1=\nabla\Pi v$, $v_2=\nabla
(I-\Pi)v$, and then with $v_1=\Pi v$, $v_2=(I-\Pi)v$, we conclude
\begin{equation*}
\re[b_h(v,v)]+2 k^2 \N{v}_{0,\Omega}^2 \ge
\min\{\alpha_s,1\}(\N{\nabla
  v}_{0,\Omega}^2+k^2\N{v}_{0,\Omega}^2)=
\min\{\alpha_s,1\}\Nh{v}^2.
\end{equation*}
\end{proof}

Under the assumption of Proposition~\ref{prop:garding}, the
G\r{a}rding inequality of Theorem~\ref{th:abstract} holds true with
$\alpha=\min\{\alpha_s,1\}$.

\subsection{Choice of $s^K(\cdot,\cdot)$}\label{sec:choice}
A first attempt in the choice of the stabilization term could be done by 
defining
\begin{equation}\label{eq:GRAD_Tr}
s^K((I-\Pi)u,(I-\Pi)v)=\int_K \nabla(I-\Pi)u\cdot\conj{\nabla(I-\Pi)v}\dV,
\end{equation}
so that (\ref{eq:stabil_hyp}) is satisfied with the equality and
$\alpha_s=1$.

We notice that the continuity of the discrete bilinear form $a_h(.,.)$ follows from the continuity of the continuous bilinear forms and the continuity of the operator $\Pi$ (see Proposition \ref{prop:contPi}).
Therefore,~\eqref{eq:contgamma} holds with 
$\gamma\ge
  1+2(\beta_{\min}^{-2}+\beta_{\min}^{-1})$, where $\beta_{\min}=\min_{K\in\calT_h}\beta(h_Kk)=\beta(hk)$.

The matrix form of this stabilization term is
\[
(\conj{I-P})^T A (I-P),
\]
where 
where $I$ is the identity matrix of appropriate size (number of directions $p$
  times number of mesh vertices $n_K$), $P$ is the
matrix representing the operator~$\Pi$ defined in Section~\ref{sec:RN}
and $A$ is the matrix of entries
\[
A(r,s)=\int_K \nabla\psi_s\cdot\conj{\nabla\psi}_r\,dV.
\]
However, when $K$ is a generic polygon, we do not resort to explicit expressions
of the basis functions $\varphi_j$ of $\VK$, and thus of $\psi_r$, but
we approximate~\eqref{eq:GRAD_Tr}.

Adapting the rationale of VEM for elliptic problem to our case, we
construct the (approximated) stabilization form as follows.
Let $\psi_r(\bx)=\varphi_j(\bx)\,\pw_{j\ell}(\bx)$ and
$\psi_s(\bx)=\varphi_\kappa(\bx)\,\pw_{\kappa m}(\bx)$  be two basis
functions in $V_p(K)$. Then
\[
\begin{split}
\nabla\psi_r&=(\nabla\varphi_j+\varphi_j\, ik\bd_\ell)\,
\pw_{j\ell},\\
\nabla\psi_s&=(\nabla\varphi_\kappa+\varphi_\kappa\, ik\bd_m)\,
\pw_{\kappa m},
\end{split}
\]
and thus
\[
\begin{split}
\nabla\psi_s\cdot\conj{\nabla\psi}_r
&=(\nabla\varphi_\kappa\cdot \nabla\varphi_j+
ik\,\varphi_\kappa\,\bd_m\cdot \nabla\varphi_j -
ik\,\varphi_j\, \bd_\ell\cdot\nabla\varphi_\kappa -
k^2 \varphi_\kappa\,\varphi_j\,\bd_m\cdot\bd_\ell)
\,\pw_{\kappa m}\,\conj{\pw}_{j \ell}.
\end{split}
\]
Taking into account the scaling of the terms in the brackets with
respect to the elemental mesh size, we neglect the last three and
replace the first by $\delta_{\kappa j}/h_K^2$.
Therefore, we define $s^K((I-\Pi)u,(I-\Pi)v)$ in terms of its
associated matrix $S_K$ as
\begin{equation}\label{eq:our_SK}
S_K=\conj{(I-P)}^T M (I-P),
\end{equation}
where $M$ is the (scaled) plane wave mass matrix 
of size $(n_K p,n_K p)$
whose entries are 
$M(r,s)=\int_K \frac{\delta_{\kappa j}}{h_K^2}\,\pw_{\kappa
  m}\,\conj{\pw}_{j\ell}\,dV$.
More precisely, if $r=(j-1)p+\ell$ and $s=(\kappa-1)p+m$,
\begin{equation}\label{eq:our_M}
\begin{split}
M(r,s)&=\frac{\delta_{\kappa
    j}}{h_K^2}\,e^{ik\bd_m\cdot(\bx_K-\bx_j)}\,e^{-ik\bd_\ell\cdot(\bx_K-\bx_\kappa)}
\int_K pw_m\,\conj{pw}_\ell\,dV\\
&=\frac{\delta_{\kappa
    j}}{h_K^2}\,e^{ik\bd_m\cdot(\bx_K-\bx_j)}\,e^{-ik\bd_\ell\cdot(\bx_K-\bx_\kappa)}
\int_K e^{ik(\bd_m-\bd_\ell)\cdot(\bx-\bx_K)}\,dV.
\end{split}
\end{equation}
We point out that the integral defining $M(r,s)$  can be computed exactly;
  see Remark~\ref{rem:exactintegrals}.

On each element $K$, the 
elemental matrix of the complete PW-VEM is therefore
\begin{equation}\label{eq:matrixform}
\conj{B}^T\conj{G}^{-1} B + \conj{(I-P)}^T M (I-P) +R
\end{equation}
where $D$, $B$, $G$ and $P$ are defined in Section~\ref{sec:RN}, $M$
is the scaled plane wave mass matrix
defined in \eqref{eq:our_M},
and $R$ is the matrix associated with the bilinear form
$ik\int_{\partial K\cap\partial\Omega} u\conj{v}\,dS$:
\[
R(r,s)=ik\int_{\partial K\cap\partial\Omega} \psi_s\conj{\psi}_r\,dS,
\qquad r,s=1,\ldots,n_Kp.
\]

\section{Numerical results}\label{sec:num}
%
Unless otherwise stated, in the following experiments
we consider the problem~\eqref{eq:helm}
in the domain $\Omega=(0,1)^2$, and with boundary
datum $g$ such that the analytical
solution is 
\[
u(\bx)=H_0^{(1)}(k\abs{\bx-\bx_0}),\qquad \bx_0=(-0.25,0),
\]
where $H_0^{(1)}$ is the zero-th order Hankel function of the first
kind. In Figure~\ref{fig:Hankel}, we report the real part of this solution
$u$ for the wave numbers $k=20$, $k=40$ and $k=60$.
All the errors reported in the following are relative errors in the $L^2$-norm.

\begin{figure}[!htbp]
  \centering
\hspace{-2truecm}
  \begin{minipage}[c]{0.33\textwidth}
    \qquad\qquad\includegraphics[width=0.90\textwidth]{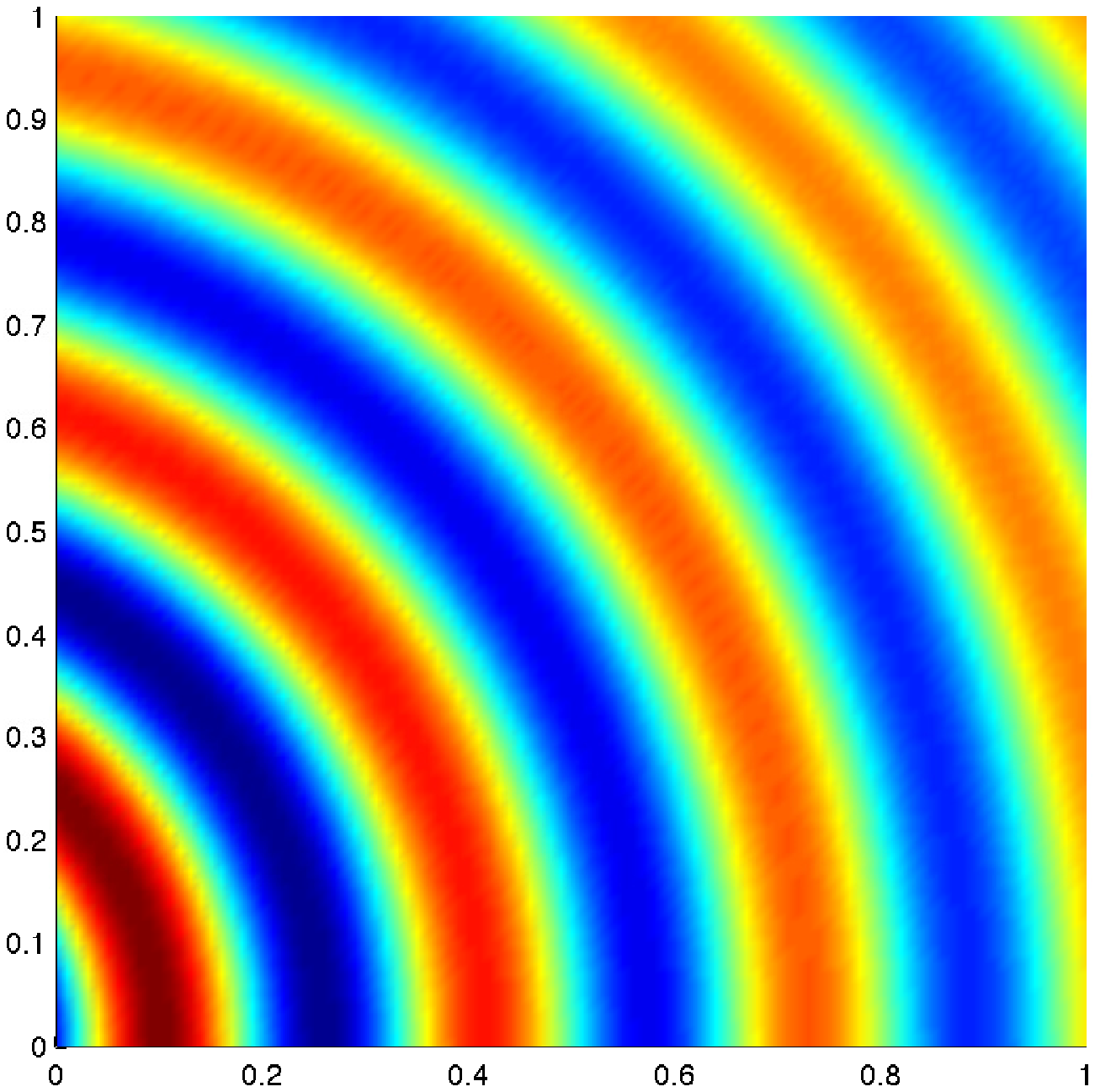}
  \end{minipage}%
  \begin{minipage}[c]{0.33\textwidth}
\qquad\qquad\includegraphics[width=0.90\textwidth]{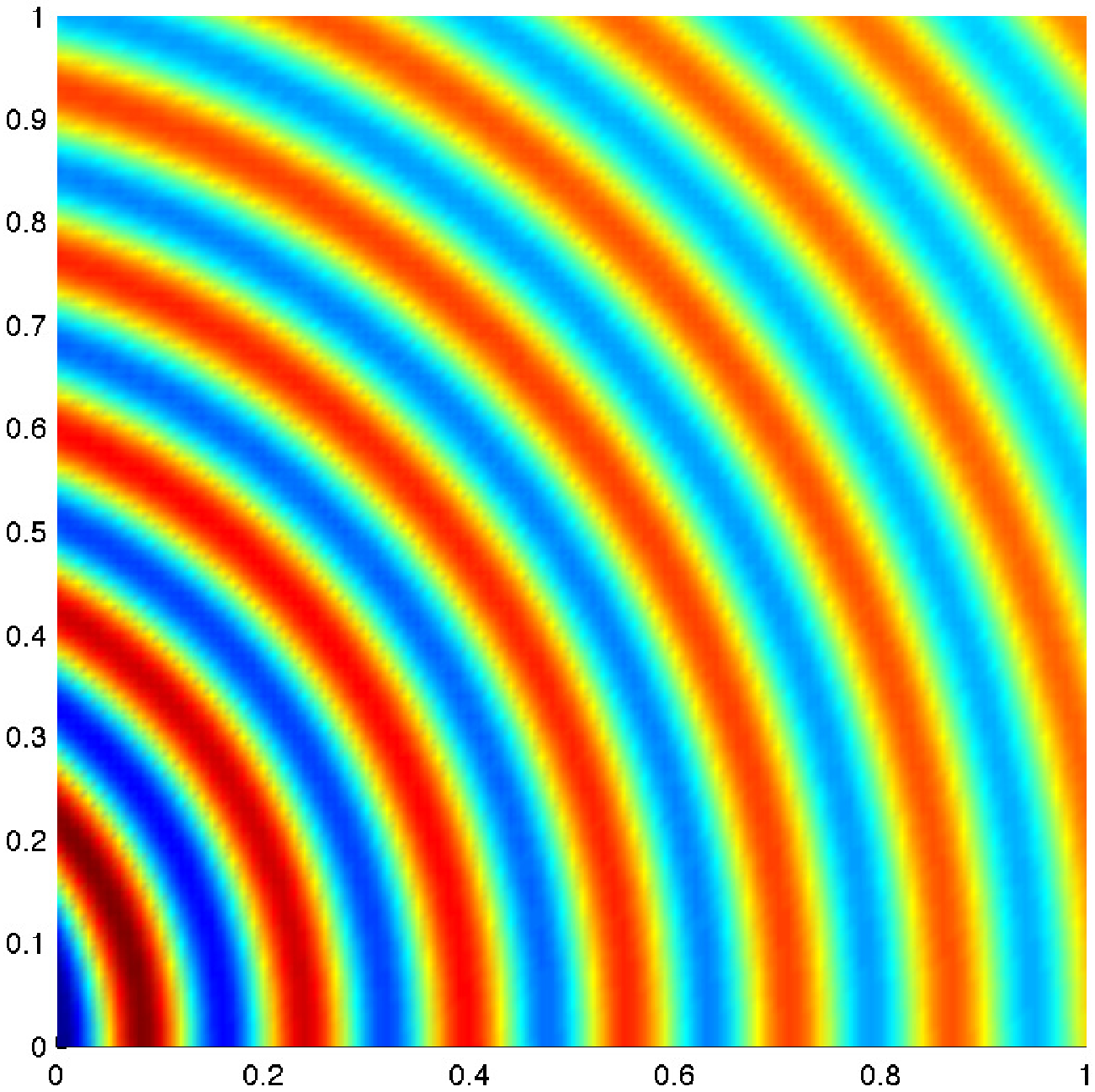}
  \end{minipage}%
  \begin{minipage}[c]{0.33\textwidth}
\qquad\qquad\includegraphics[width=0.90\textwidth]{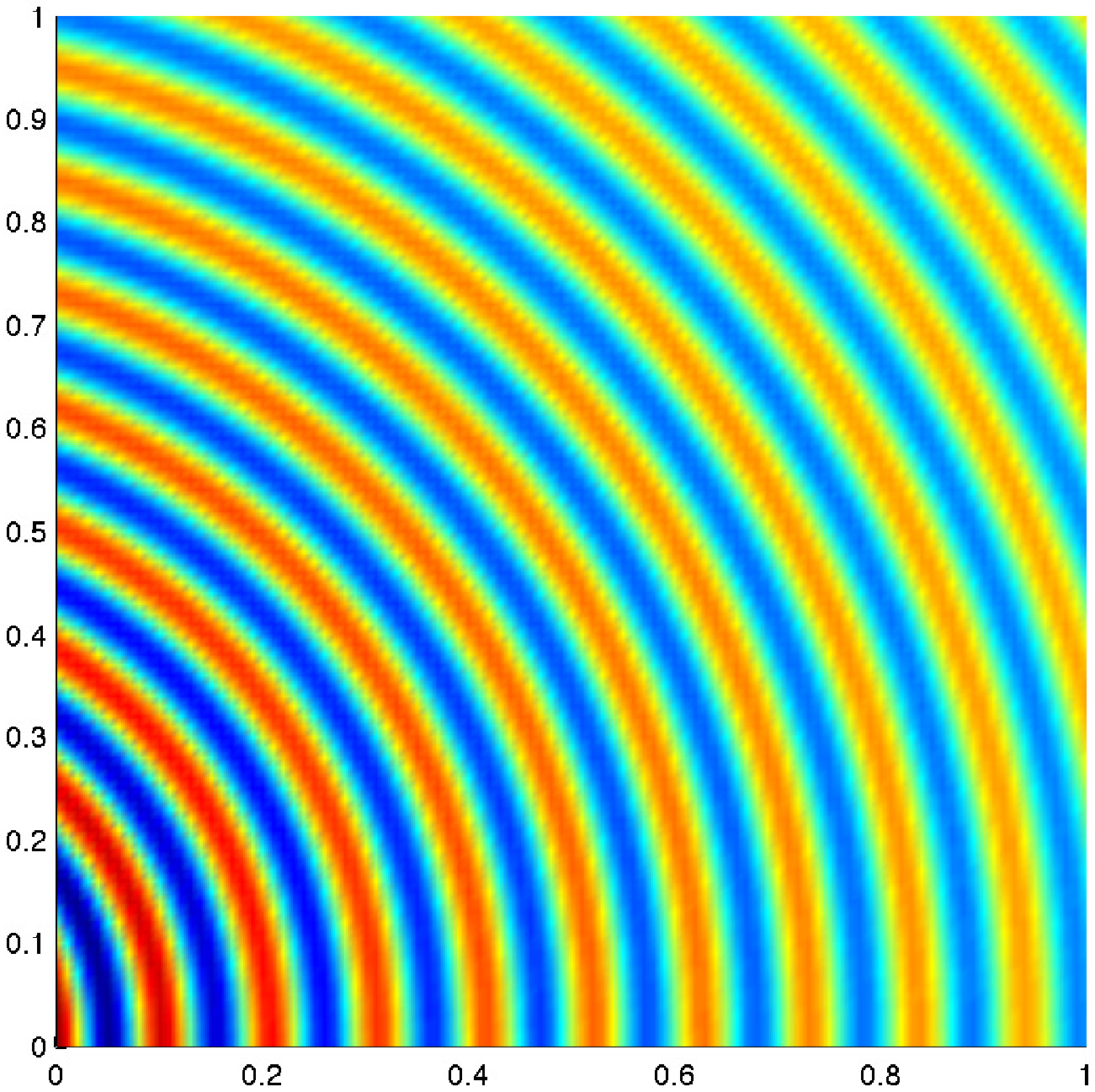}
  \end{minipage}
\caption{Real part of the function
  $u(\bx)=H_0^{(1)}(k\abs{\bx-\bx_0})$, $\bx_0=(-0.25,0)$ in the
  domain $\Omega=(0,1)^2$, for $k=20$ (left), $k=40$ (center), and $k=60$ (right).}
  \label{fig:Hankel}
\end{figure}

The function $u(\bx)$ is a regular solution of the
homogeneous Helmholtz problem, so that the $h$-convergence 
result 
of Corollary \ref{cor:rates} provides 
order $h^m$ (with $m=(p-1)/2$) for the error in the $\Nh{.}$ norm. 
An $L^2$-error bound 
is then given by \eqref{eq:L2-error}. Because of the
restriction \eqref{eq:threshold} on the mesh size, we have 
$hk^2<C$
and 
the factor $hk$ 
in \eqref{eq:L2-error} implies the gain of a factor 
$h^{1/2}$ 
for the $L^2$-error, with respect to the $\Nh{.}$-error.

Due to the fact that high-order approximation properties of
  the plane wave spaces hold true only for solutions to the
  homogeneous Helmholtz equation, and our duality argument in the proof
  of Theorem~\ref{th:abstract} makes use of a dual problem with non
  zero right-hand side (see equation~\eqref{eq:dual}), our analysis
  does not cover the asymptotics in $p$. On the other hand,
the $p$-version of the best approximation results 
(see~\cite[Th.~3.9, Rem.~3.13]{PVersion} and
  \cite[Sect.~5.2]{ExpConv} for the best approximation estimates in
$V_p^{\ast}(\calT_h)$; for $V_p(\calT_h)$, use again
\cite[Th. 2.1]{BAM96}) give algebraic convergence
  $(\log(p)/p)^\ell$, provided that the exact solution belongs to
  $H^{\ell+1}(\Omega)$ and $p$ is large enough, or exponential convergence,
whenever the exact solution can be extended analytically outside $\Omega$.
Therefore, we also present numerical results
for fixed $h$ and increasing $p$. 


First, we numerically evaluate in Section~\ref{sec:numerics_PUMGRAD}
the loss of accuracy due to the approximation of the stabilization
term. Then, in Section~\ref{sec:numerics_conv}, we test the $h$- and $p$-versions of the PW-VEM on
polygonal meshes, and compare the results obtained for different
values of the wave number $k$.
We conclude by testing the $p$-convergence for non analytic
  exact solutions (Section~\ref{sec:nonsmooth}).

\subsection{Effects of the approximation of the stabilization
  term}\label{sec:numerics_PUMGRAD}

We recall that the choice $s^K((I-\Pi)u,(I-\Pi)v)=a^K((I-\Pi)u,(I-\Pi)v)$ in the
formulation~\eqref{def:ah} coincides with the partition of unity (PUM)
method, since the complete bilinear form is considered in this case.
On triangular meshes, where the canonical VEM basis functions
$\varphi_j$ coincide with the classical Lagrangian $\IP_1$ basis
functions, we compare the error of the PUM, with that of the VEM
with the stabilization~\eqref{eq:GRAD_Tr} (GRAD in the
following) and that of the PW-VEM, whose elemental matrices
are given by~\eqref{eq:matrixform}.

We have used structured triangular meshes, obtained from Cartesian
subdivisions of $\Omega$, dividing then each square into two triangles
by one of the diagonals. 
We report in Table~\ref{tab:COMPhversion} the results obtained on four
meshes containing, respectively, 8, 32, 128 and 512 triangles, using
$p=13$ plane waves per node, for the wave number $k=20$.
When the mesh size $h$ is sufficiently small, GRAD is very close to PUM, showing that 
the sufficient requirement \eqref{eq:stabil_hyp} on the stabilization term is 
indeed sharp. 
When the approximation of the stabilization term defined in~\eqref{eq:our_SK}--\eqref{eq:our_M}
is used instead (PW-VEM),
some accuracy is lost. However, the same order of
convergence as for PUM is maintained.

\begin{table}[h]
\begin{center}
\begin{tabular}{|c| |c|c| |c|c| |c|c|}
   \hline
&\multicolumn{2}{c||}{PUM}&\multicolumn{2}{c||}{GRAD}&\multicolumn{2}{c|}{PW-VEM}\\
{$h$}&$L^2$-error & rate &$L^2$-error & rate&$L^2$-error & rate\\
\hline
7.0711$e-$01&1.9213$e-$02&-         &7.1989$e-$01&-         &4.1548$e-$01&-\\
3.5355$e-$01&4.0683$e-$04&5.5615&1.5517$e-$03&8.8577&1.0990$e-$02&5.2406 \\
1.7678$e-$01&3.4126$e-$06&6.8974&3.3981$e-$06&8.8349&1.2969$e-$04&6.4050\\
8.8388$e-$02&4.0164$e-$08&6.4088&4.0148$e-$08&6.4033&1.1089$e-$06& 6.8698\\
\hline
\end{tabular}
\end{center}
\caption{Relative error in the $L^2$-norm for PUM, GRAD, PW-VEM for structured triangular
  meshes; $k=20$, $p=13$.}
\label{tab:COMPhversion} 
\end{table}


\begin{figure}[!htbp]
  \centering
  \begin{minipage}[c]{0.50\textwidth}
\psfrag{p}[c]{\scriptsize $p$}
\psfrag{error}[c]{\scriptsize $\log(L^2\text{-error})$}
    \qquad\qquad\includegraphics[width=0.75\textwidth]{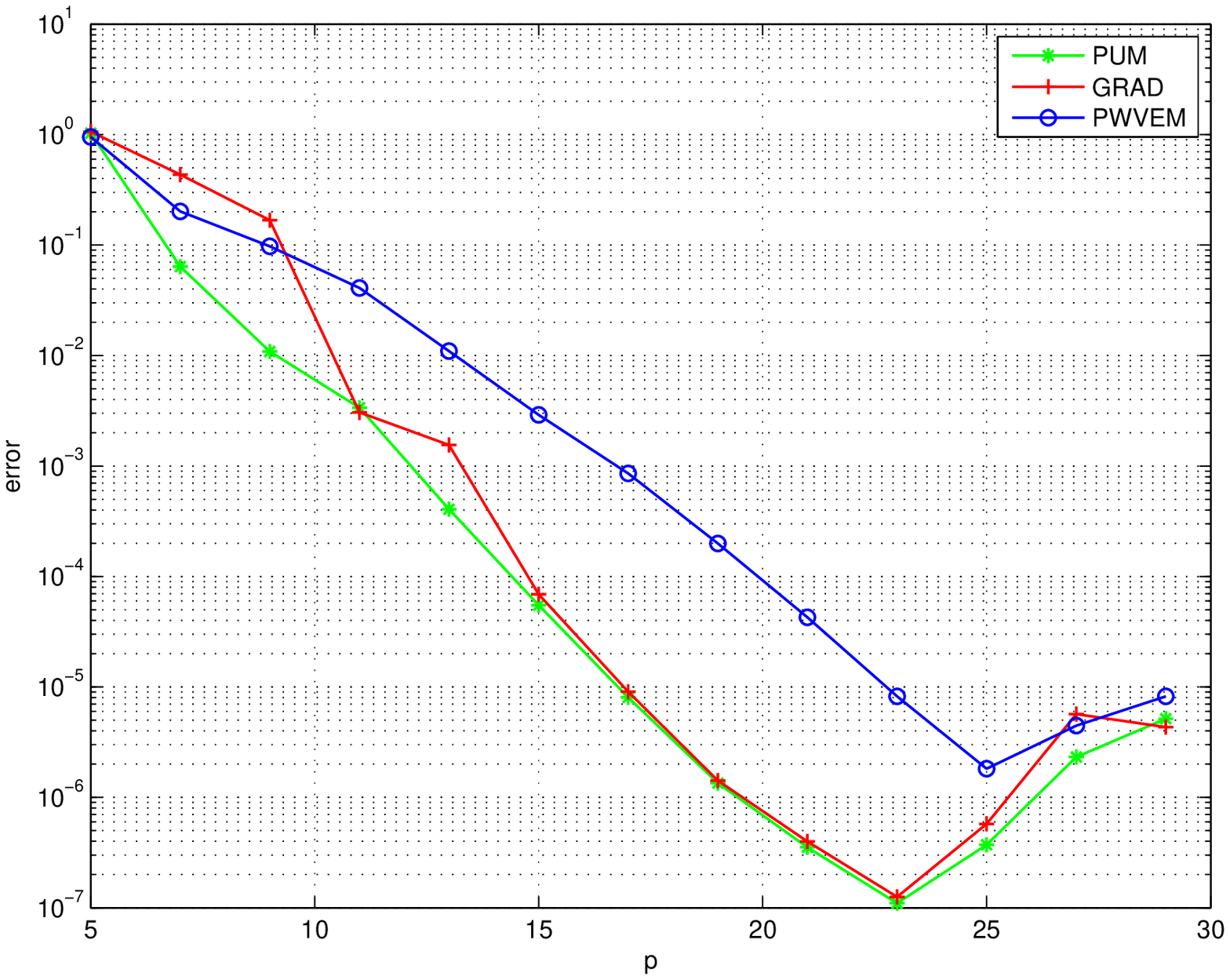}
  \end{minipage}%
  \begin{minipage}[c]{0.50\textwidth}
\psfrag{p}[c]{\scriptsize $p$}
\psfrag{error}[c]{\scriptsize $\log(L^2\text{-error})$}
    \qquad\includegraphics[width=0.75\textwidth]{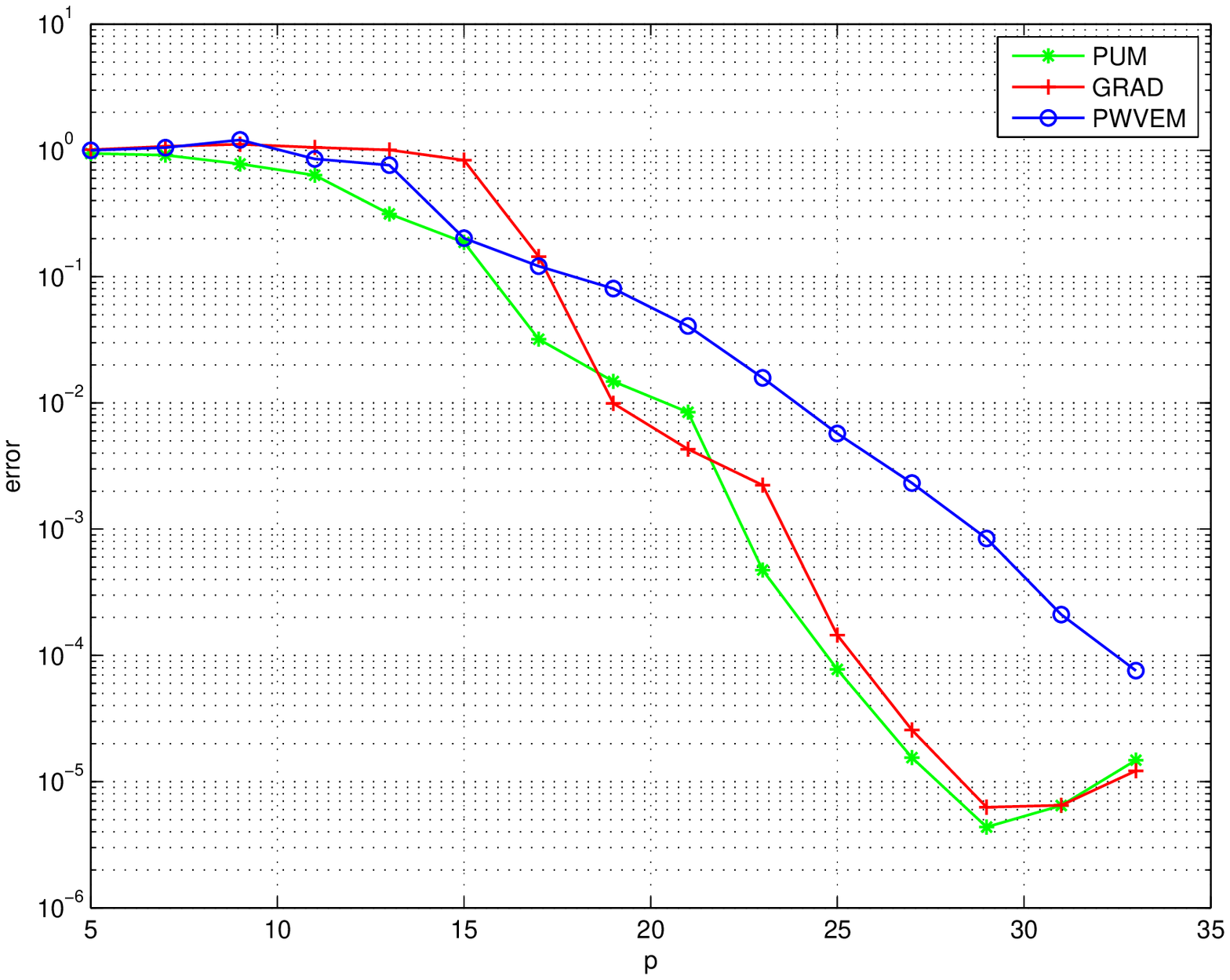}
  \end{minipage}
\caption{Relative error in the $L^2$-norm for PUM, GRAD, PW-VEM for the structured triangular
  mesh with 32 elements ($h=$3.5355$e-$01), for different values
  of $p$, for $k=20$ (left) and $k=40$ (right).}
\label{fig:COMPpversion} 
\end{figure}

We also report in Figure~\ref{fig:COMPpversion}
the errors with the three methods on the mesh with 32 elements,
varying $p$, for $k=20$ and $k=40$,
respectively. 
As in the $h$-error study, when the best
  approximation error is sufficiently small (here, the number of plane
  waves per node is sufficiently large), GRAD coincides with PUM, and
  PW-VEM looses some accuracy, but still showing an exponential convergence
  behavior, with a comparable order of convergence. Notice that, when
  increasing the wave number $k$, all these methods exhibit a larger
  preasymptotic region with slower convergence. Moreover, in both
  figures, and in many of the next plots, 
for large values of $p$, instability takes place due to the
  ill-conditioning of the plane wave basis, and the impact of roundoff
  error results in the increasing of the error;
we refer to~\cite{PVersion} for a similar behavior of the plane wave
discontinuous Galerkin (PWDG) method. We remark that a similar
phenomenon occurs for small $h$ (see Section~\ref{sec:numerics_conv} below).

We point out that the tests presented above do not provide 
a complete comparison
  between the PW-VEM and PUM (and/or GRAD), since we restricted to
  triangular meshes, where PUM can be exactly computed. A fair
  comparison should be carried out on more general meshes, where the use of PUM with generalized barycentric
  coordinates would need some quadrature formulas (we refer 
to~\cite{MRS2014}, where a PUM with generalized barycentric
  coordinates  is compared to VEM with different stabilizations 
for the Poisson
  problem). 
On the other hand, these tests on triangular meshes suggest that there is a margin of improvement for the choice of the approximation
of the stabilization term defined
in~\eqref{eq:our_SK}--\eqref{eq:our_M}.

\subsection{Convergence}\label{sec:numerics_conv}

We test convergence of the PW-VEM on (polygonal)
Voronoi meshes made of $2^n$ elements, $3\le n\le 9$.
We report in Figure 3 the meshes with 16 and 64 elements.


\begin{figure}[!htbp]
  \centering
  \begin{minipage}[c]{0.50\textwidth}
   \qquad\qquad\includegraphics[width=0.78\textwidth]{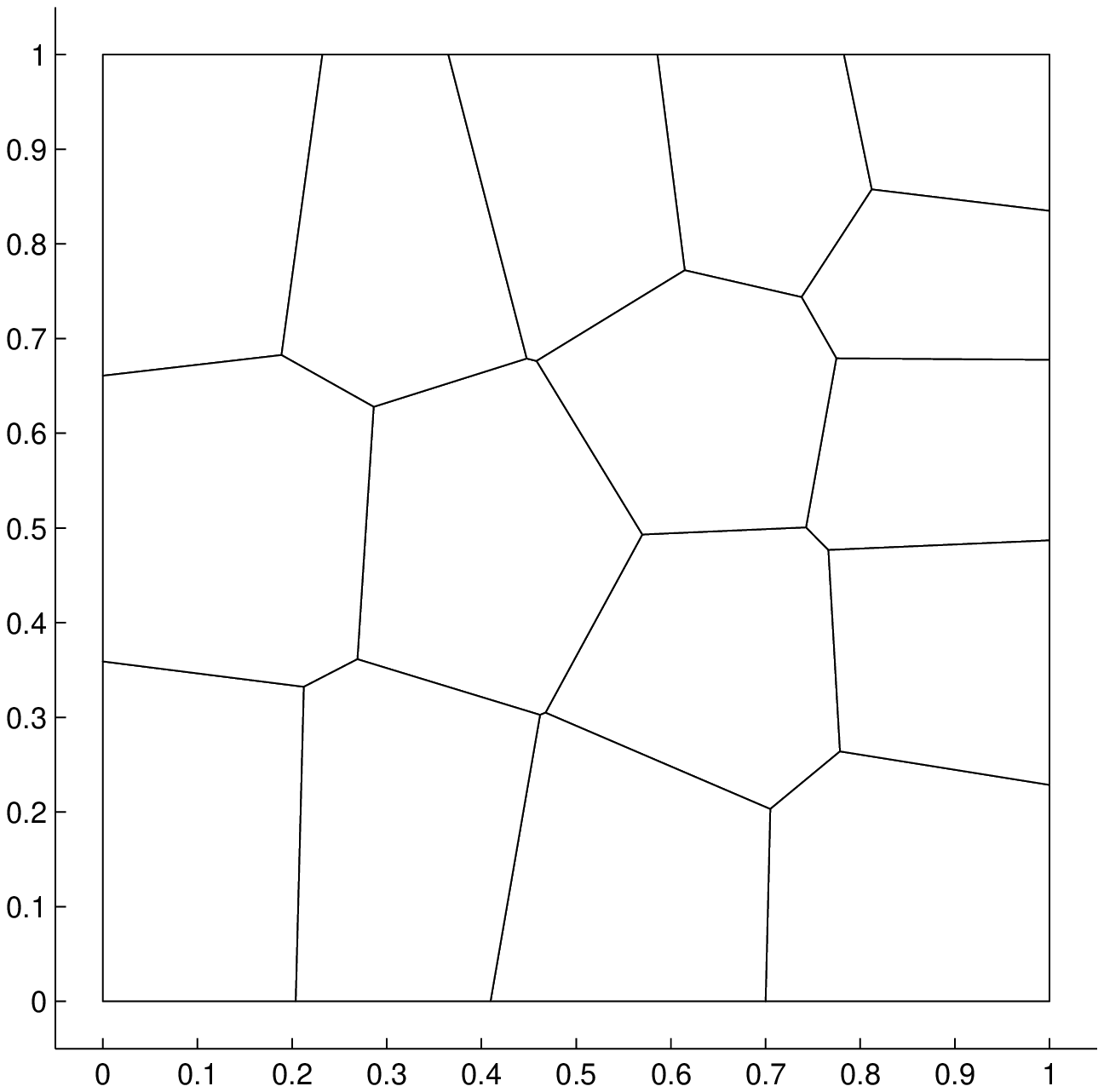}
  \end{minipage}%
  \begin{minipage}[c]{0.50\textwidth}
    \includegraphics[width=0.78\textwidth]{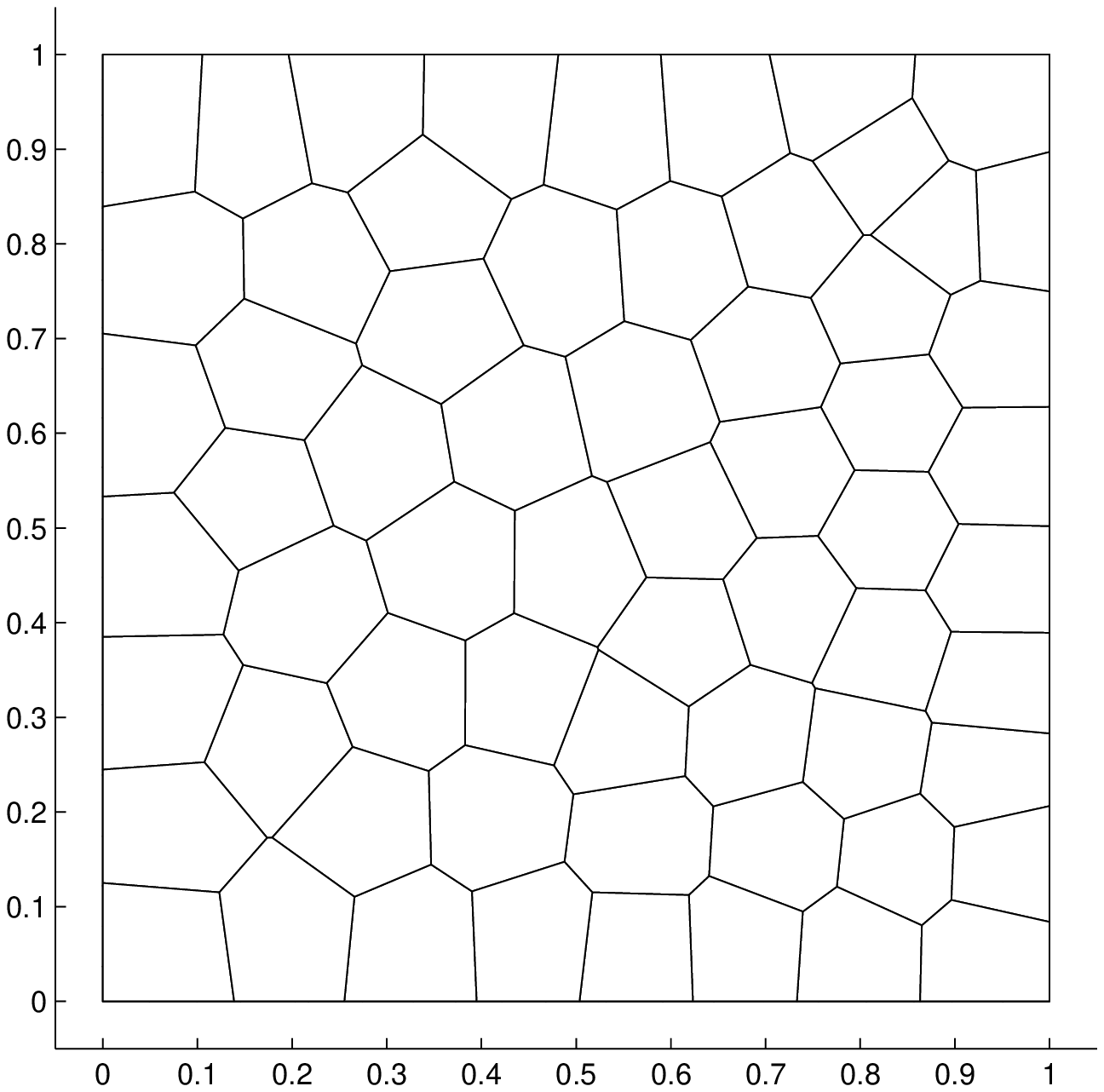}
  \end{minipage}
\caption{Voronoi meshes with 16 (left) and 64 (right) elements.}
  \label{fig:Voronoi}
\end{figure}

\subsubsection{\bf $h$-convergence}\label{sec:hconv}

We test first the $h$-convergence. 
We report in Figure~\eqref{fig:hconvk} the relative errors in the
$L^2$-norm of the PW-VEM method on the Voronoi meshes,
for $p=13$ and $k=20$, $40$, $60$. 
As pointed out at the beginning of Section~\ref{sec:num}, the
  expected convergence rate is 
$m+1/2$, i.e., $6.5$ for $p=13$. This rate seems to be actually reached,
after a pre-asymptotic region, which is wider for larger
$k$. Also the results of Table~\ref{tab:hconvergence} seem to confirm
these rates ($6.5$ for $p=13$ and $7.5$ for $p=15$).
However, since the sequence of Voronoi meshes considered here is not made by nested meshes, 
the maximum  of the diameters of the elements $K \in \calT_h$
(definition of $h$ for  Table~\ref{tab:hconvergence}) does not vary by
a factor from one mesh to the other, so that a precise assessment of
the convergence rate is difficult. Nevertheless, we can conclude that
high order rate is provided also on polygonal meshes.
We also report in Figure~\ref{fig:hconvp} the error plots for
  $k=20$ and for different values of $p$. For fine meshes and high
  $p$, instability takes place and the error increases, as already observed.

\begin{figure}[!htbp]
  \centering
\psfrag{invh}[c]{\scriptsize $1/h$}
\psfrag{error}[c]{\scriptsize $\log(L^2\text{-error})$}
    \includegraphics[width=0.50\textwidth]{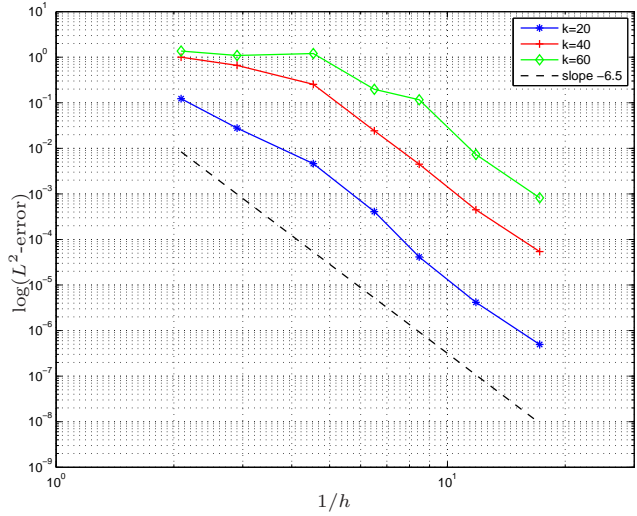}
\caption{Relative error in the $L^2$-norm of the PW-VEM 
for $p=13$ and $k=20$, $40$, $60$ on the Voronoi meshes.}
  \label{fig:hconvk}
\end{figure}

\begin{table}[h]
\begin{center}
\begin{tabular}{|c| |c|c| |c|c|}
   \hline
&\multicolumn{2}{c||}{$p=13$}&\multicolumn{2}{|c|}{$p=15$}\\
{$h$}&$L^2$-error & rate &$L^2$-error & rate\\
\hline
3.4487$e-$01&2.7882$e-$02& -
&1.1374$e-$02&-
\\
2.2017$e-$01&4.6014$e-$03&4.0147   &1.5253$e-$03&4.4772\\
1.5376$e-$01&4.0962$e-$04&6.7374   &6.6821$e-$05&8.7123\\
1.1806$e-$01&4.1264$e-$05&8.6874   &5.3076$e-$06&9.5868\\
8.4571$e-$02&4.1597$e-$06&6.8785   &9.3361$e-$07&5.2096\\
\hline
\end{tabular}
\end{center}
\caption{Relative error in the $L^2$-norm for PW-VEM for $k=20$, with $p=13$ and $p=15$, on Voronoi meshes.}
\label{tab:hconvergence} 
\end{table}

\begin{figure}[!htbp]
  \centering
\psfrag{ndof}[c]{\scriptsize $\#\, \text{d.o.f.}$}
\psfrag{error}[c]{\scriptsize $\log(L^2\text{-error})$}
    \includegraphics[width=0.50\textwidth]{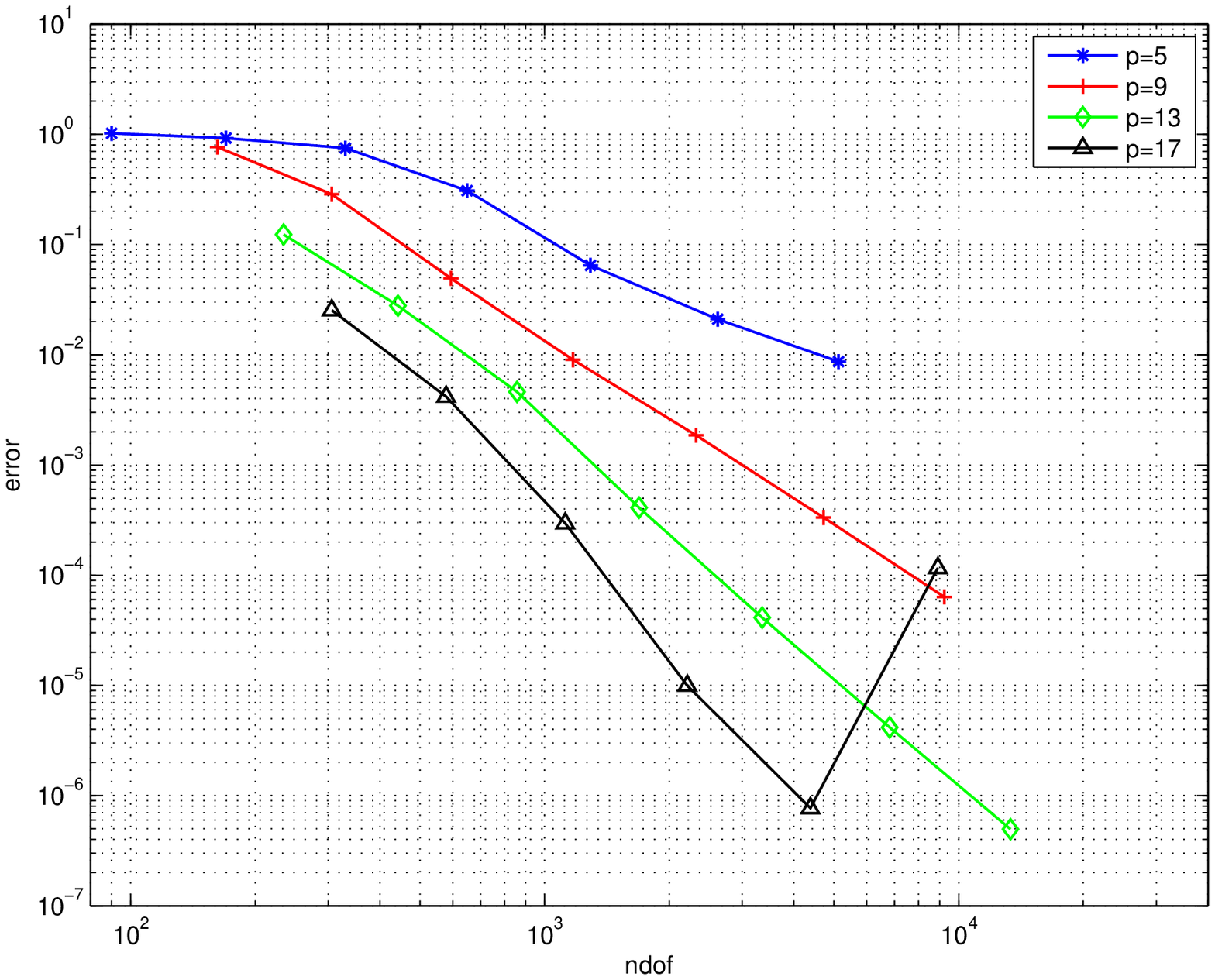}
\caption{Relative error in the $L^2$-norm of the PW-VEM 
for $k=20$ and different values of $p$ on the Voronoi meshes.}
  \label{fig:hconvp}
\end{figure}

In order to test the pollution effect, we have run the
PW-VEM, 
with $p=9$, on the Voronoi meshes made of $2^n$ elements, $3\le n\le 9$,
and $k$ chosen such that the product $hk$ is the constant 3
(i.e., $k=$ 6.26, 8.70, 13.63, 19.51, 25.41, 35.47, 51.59,
respectively).
The results reported in Figure~\ref{fig:pollution} confirm that the
$h$-version of the PW-VEM is affected by pollution, i.e., 
the boundedness of
the product $hk$ 
is not sufficient to guarantee convergence of the discretization
error. 


\begin{figure}[!htbp]
  \centering
\psfrag{invh}[c]{\scriptsize $1/h$}
\psfrag{error}[c]{\scriptsize $\log(L^2\text{-error})$}
    \includegraphics[width=0.50\textwidth]{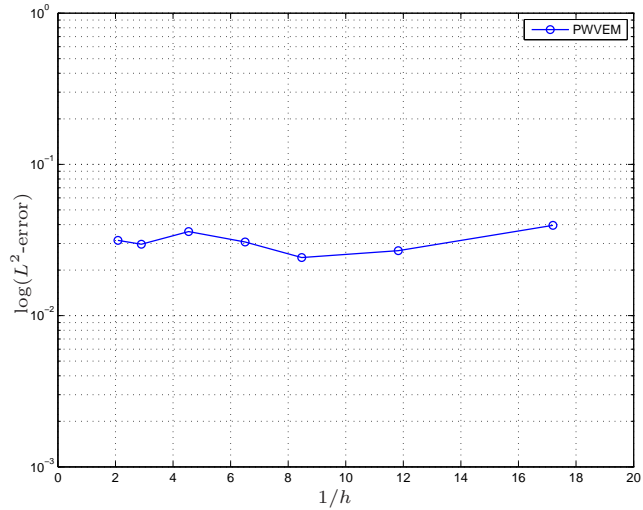}
\caption{Relative error in the $L^2$-norm of the PW-VEM 
on the Voronoi meshes for $hk=3$ and for $p=9$.}
  \label{fig:pollution}
\end{figure}

\subsubsection{\bf $p$-convergence}\label{sec:pconv}

In Figure~\ref{fig:pconvergence}, we report the relative
  errors in the $L^2$-norm for different values of $p$ in the case
  $k=20$, on the Voronoi meshes
with 16 and 64 elements.
We see exponential convergence in $p$, before instability
  for high $p$ takes place, as in Figure~\ref{fig:COMPpversion}, left,
  for the triangular meshes. 
In the case of larger elements (Figure~\ref{fig:pconvergence}, left), the pre-asymptotic region is wider,
while for smaller elements (Figure~\ref{fig:pconvergence}, right), the
instability starts to take place for smaller values of $p$.
Notice that the considered meshes also contain edges which are small, as
compared to the mesh size. These results, together with those in Section~\ref{sec:hconv}, thus suggest that the PW-VEM
is robust in case of degenerating sides.

\begin{figure}[!htbp]
  \centering
  \begin{minipage}[c]{0.50\textwidth}
\psfrag{p}[c]{\scriptsize $p$}
\psfrag{error}[c]{\scriptsize $\log(L^2\text{-error})$}
    \qquad\qquad\includegraphics[width=0.75\textwidth]{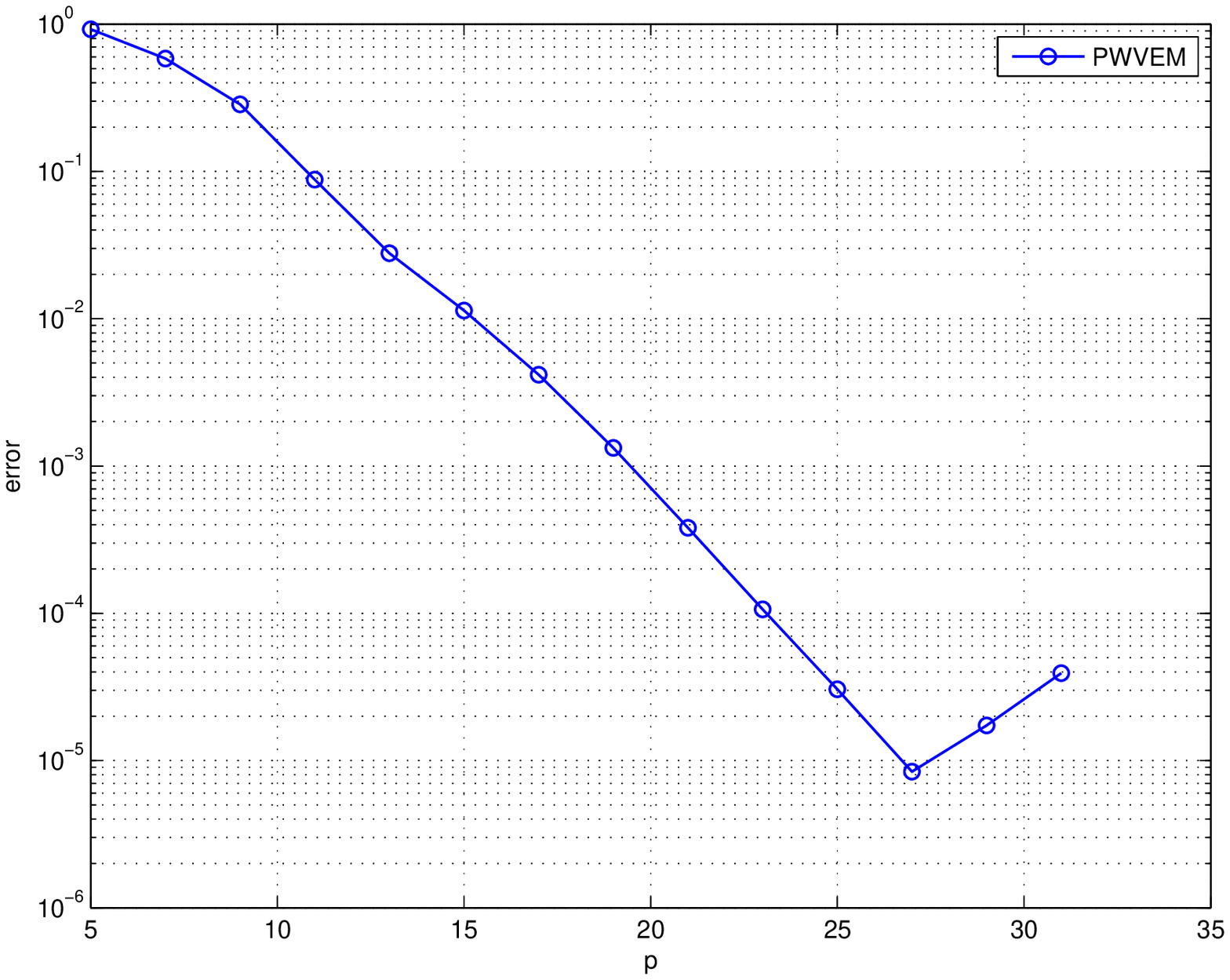}
  \end{minipage}%
  \begin{minipage}[c]{0.50\textwidth}
\psfrag{p}[c]{\scriptsize $p$}
\psfrag{error}[c]{\scriptsize $\log(L^2\text{-error})$}
    \qquad\includegraphics[width=0.75\textwidth]{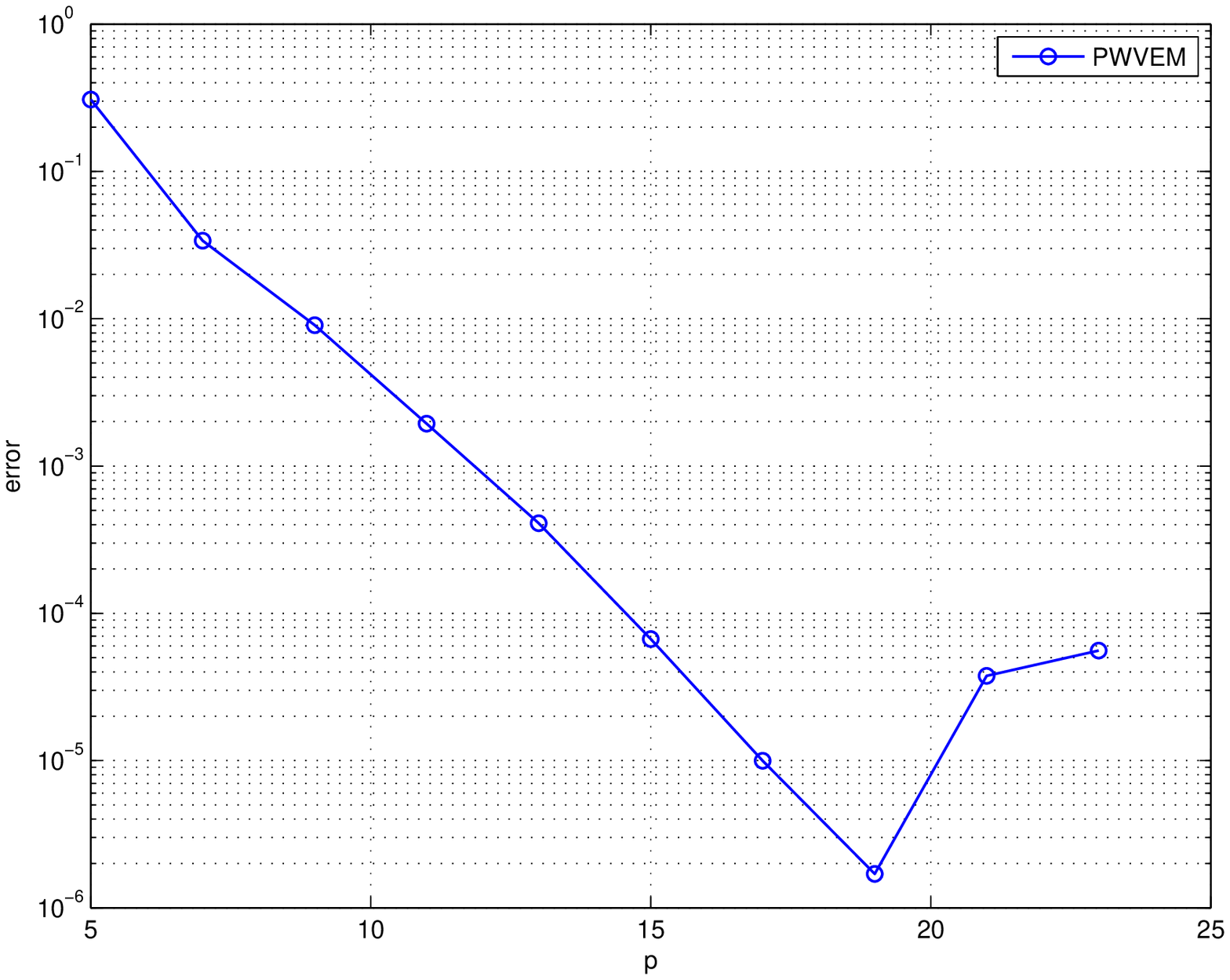}
  \end{minipage}
\caption{Relative error in the $L^2$-norm for PW-VEM for $k=20$ on
Voronoi meshes with 16 (left) and 64 (right) elements, for different
values of $p$.}
  \label{fig:pconvergence}
\end{figure}

In Figure~\ref{fig:nonconvex}, a polygonal mesh with non convex elements (100
elements) is represented, and the corresponding $L^2$-error for $k=20$
and different values of $p$ is plotted. The results are
  comparable to those obtained with meshes with convex elements,
  confirming robustness of the method also in the presence of non
  convex elements.
Due to the presence of
small elements, instability starts to take place for smaller values of
$p$, as compared to the previous, more regular, meshes.

\begin{figure}[!htbp]
  \centering
  \begin{minipage}[c]{0.50\textwidth}
    \qquad\includegraphics[width=0.78\textwidth]{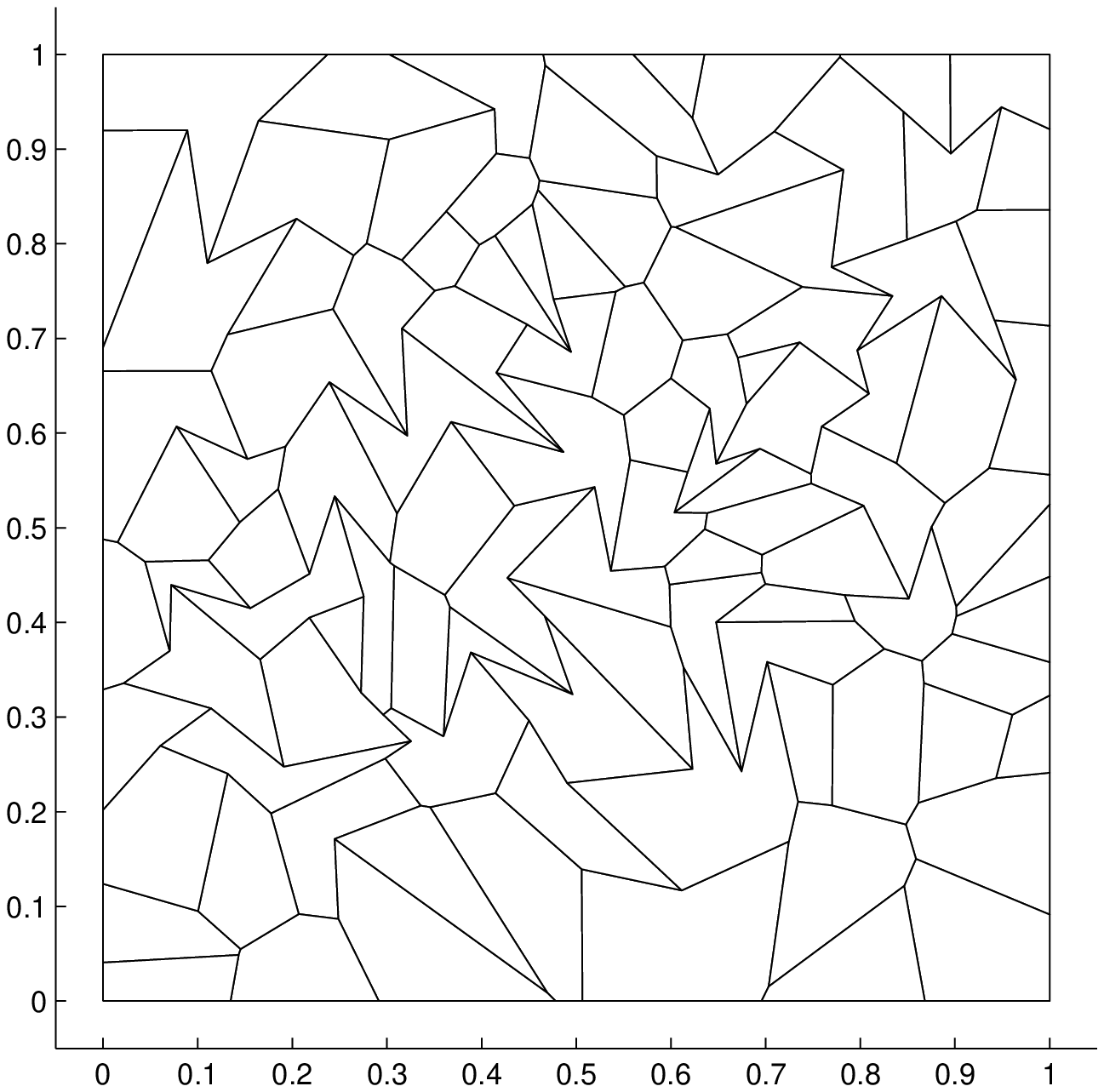}
  \end{minipage}%
  \begin{minipage}[c]{0.50\textwidth}
\vspace{0.5truecm}
\psfrag{p}[c]{\scriptsize $p$}
\psfrag{error}[c]{\scriptsize $\log(L^2\text{-error})$}
    \qquad\includegraphics[width=0.75\textwidth]{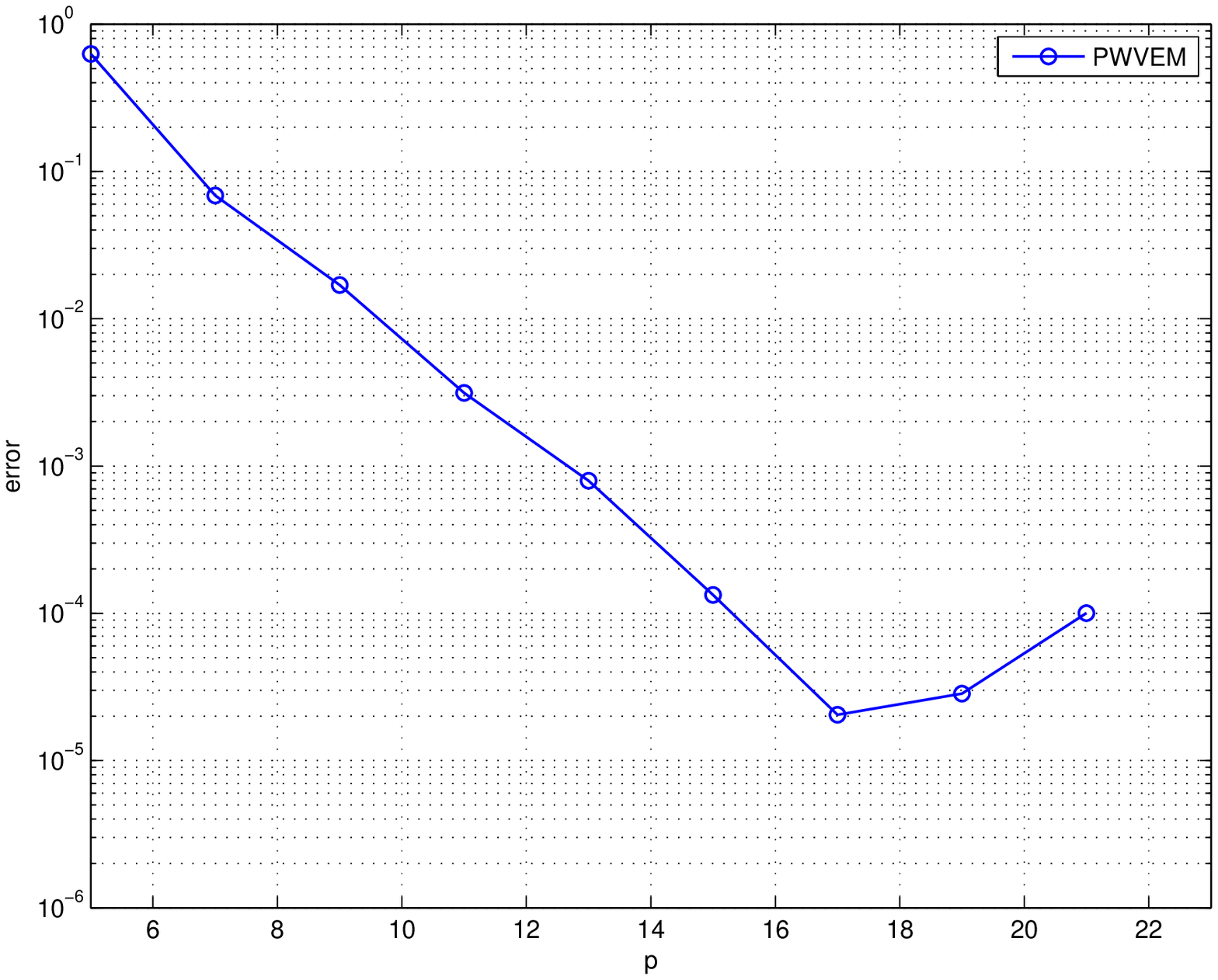}
  \end{minipage}
\caption{Polygonal mesh with 100 elements containing non convex
  elements (left), and 
relative error in the $L^2$-norm for PW-VEM for $k=20$ and different values of $p$ (right).}
  \label{fig:nonconvex}
\end{figure}

We compare now the convergence behaviour of the PW-VEM for different
values of the wavenumber $k$. We plot in Figure~\ref{fig:differentk}
the $L^2$-error of the PW-VEM on the Voronoi mesh with 64 element, for
different values of $p$ and for $k=20,\,40,\,60$. For larger values of
  $k$, the pre-asymptotic region is wider, while instability starts to
  take place for larger values of $p$.


\begin{figure}[!htbp]
  \centering
\psfrag{p}[c]{\scriptsize $p$}
\psfrag{error}[c]{\scriptsize $\log(L^2\text{-error})$}
    \includegraphics[width=0.50\textwidth]{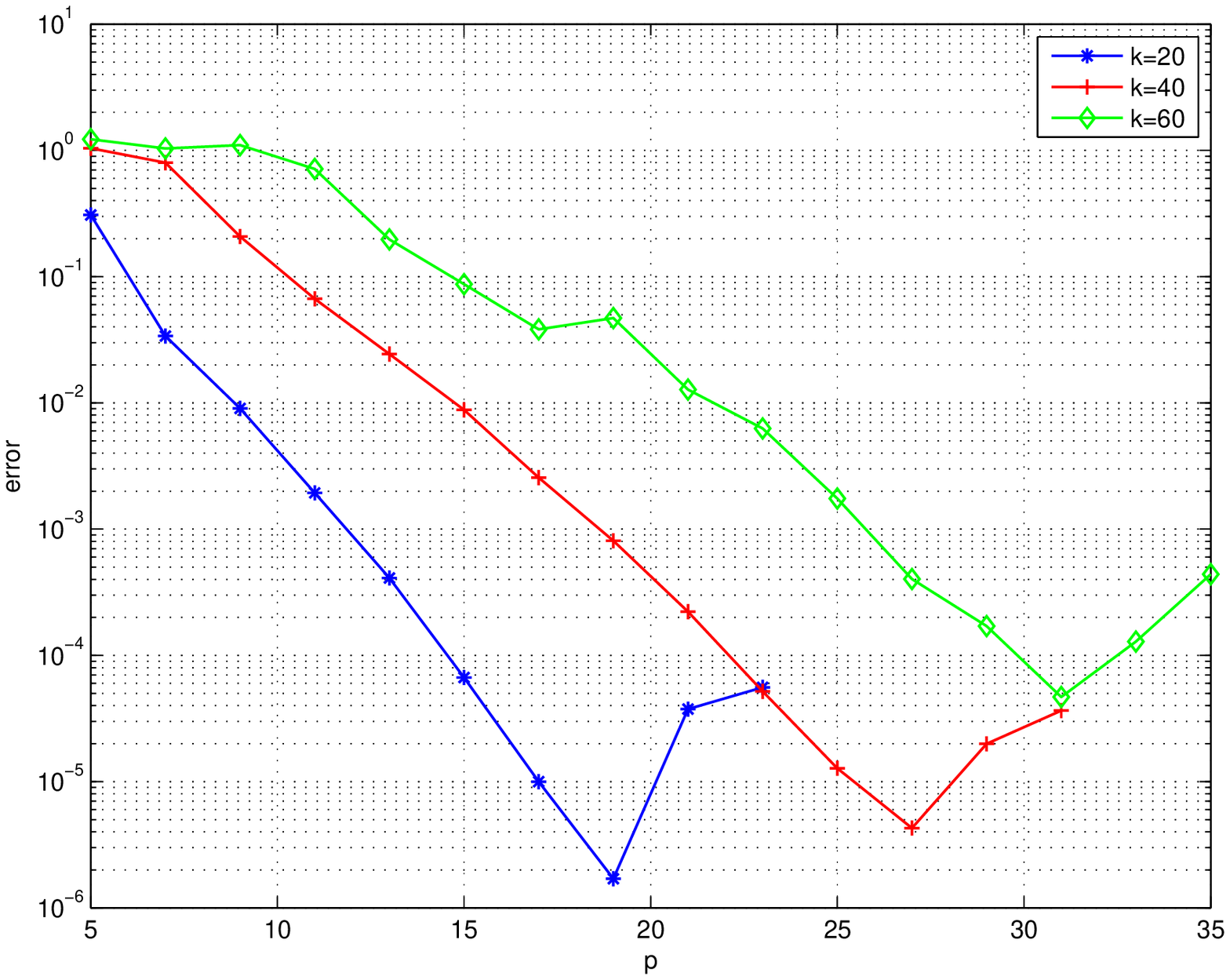}
\caption{Relative error in the $L^2$-norm of the PW-VEM on the Voronoi mesh with 64 elements, for
different values of $p$ and for $k=20,\,40,\,60$.}
  \label{fig:differentk}
\end{figure}




\subsubsection{\bf $p$-convergence for non smooth solutions}\label{sec:nonsmooth}

Finally, we test 
the $p$-convergence in the case of solutions with
  low regularity. To this aim,
as in \cite[Sect~4]{PVersion}, we consider now the problem~\eqref{eq:helm} again
in the domain $\Omega=(0,1)^2$, and with boundary
datum $g$ such that the analytical
solution is given, in polar coordinates $\bx=(r\cos\vartheta,r\sin\vartheta)$, by
\begin{equation}\label{eq:singsol}
u(\bx-\bx_0)=J_\xi(kr)\cos(\xi\vartheta), \qquad\xi\ge 0,\ \bx_0=(0,0.5),
\end{equation}
where $J_\xi$ is the Bessel function of the first
kind and order $\xi$. 
We choose $k=10$.
For integer $\xi$, $u$ can be extended
analytically outside $\Omega$, 
otherwise its derivatives have a singularity at $\bx_0$. We report in
Figure~\ref{fig:sing} the real part of $u$ for $\xi=1$ (regular
case), $\xi=3/2$ ($u\in H^2(\Omega)$) and $\xi=2/3$ ($u\not\in H^2(\Omega)$).

\begin{figure}[!htbp]
  \centering
\hspace{-2truecm}
  \begin{minipage}[c]{0.33\textwidth}
    \qquad\qquad\includegraphics[width=0.90\textwidth]{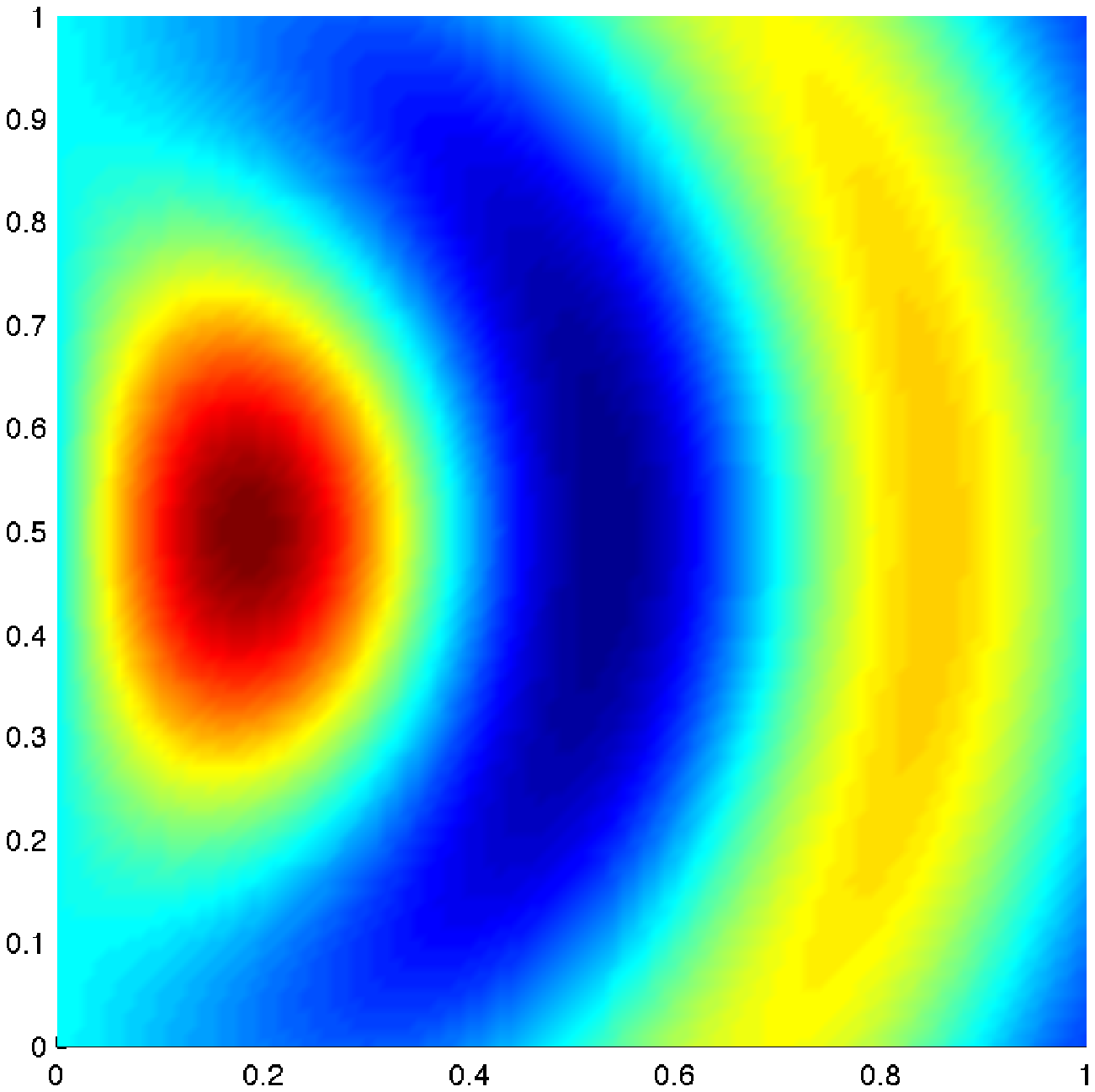}
  \end{minipage}%
  \begin{minipage}[c]{0.33\textwidth}
\qquad\qquad\includegraphics[width=0.90\textwidth]{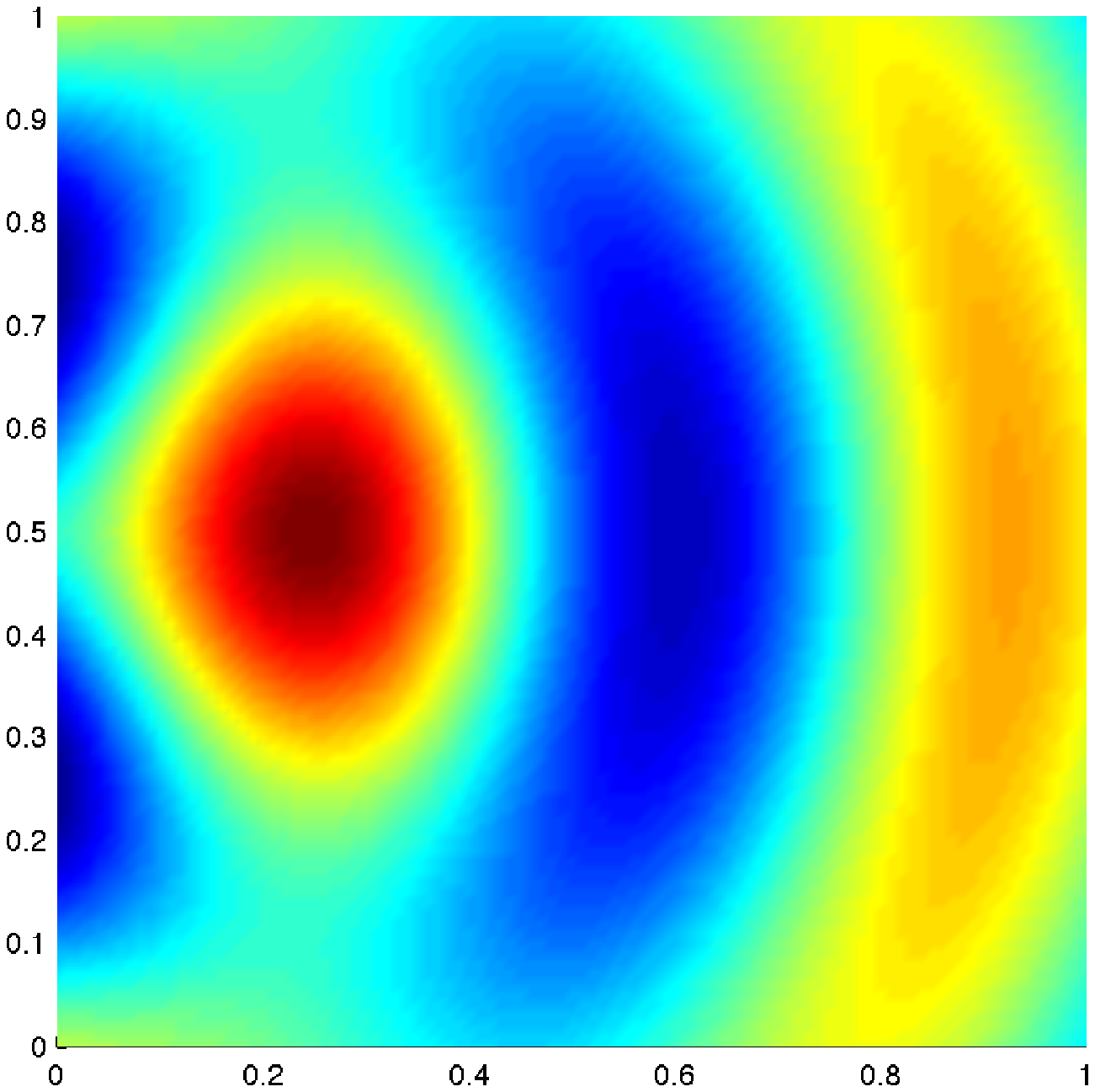}
  \end{minipage}
  \begin{minipage}[c]{0.33\textwidth}
\qquad\qquad\includegraphics[width=0.90\textwidth]{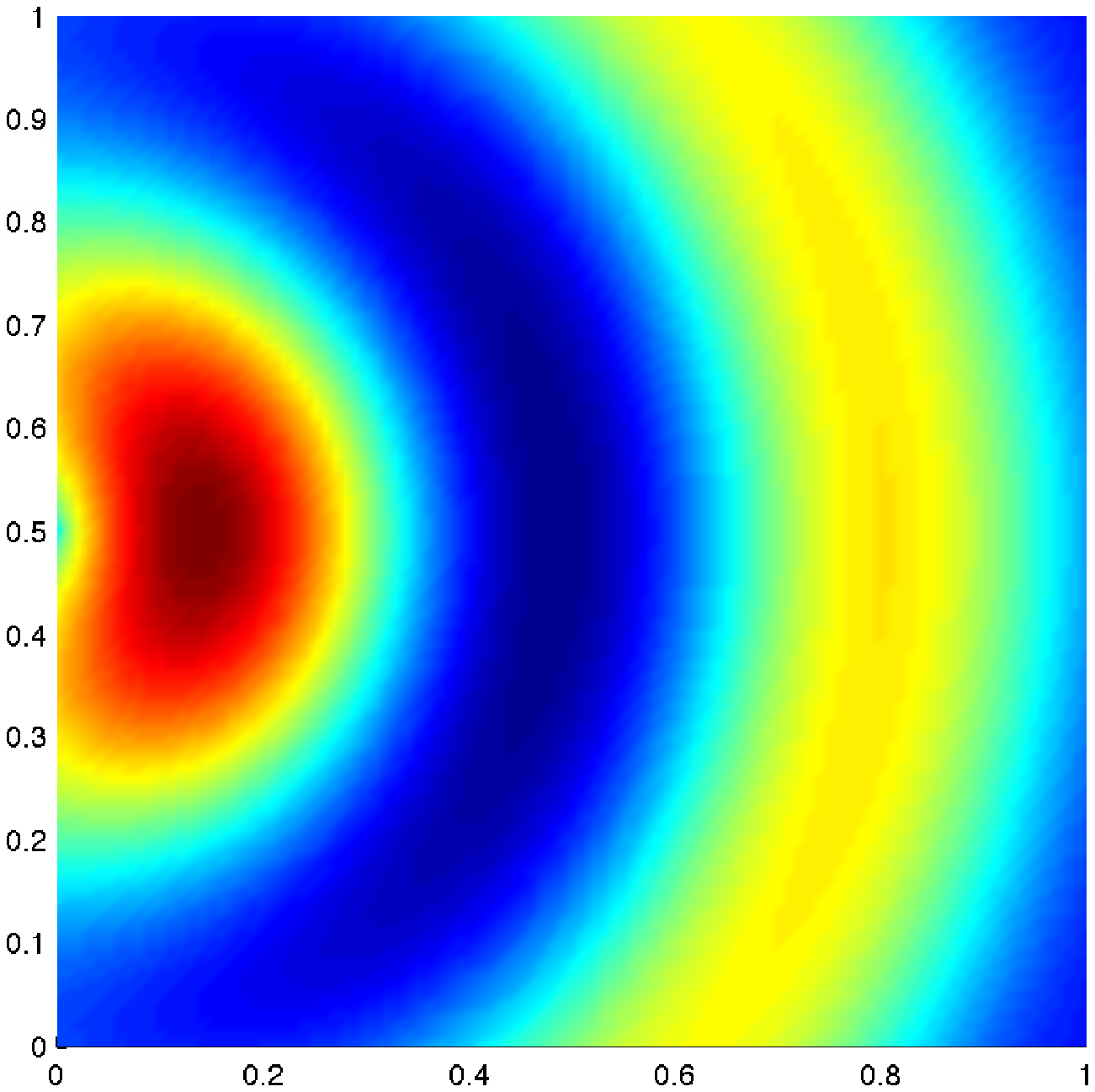}
  \end{minipage}
\caption{Real part of the function
  $u(\bx-\bx_0)=J_\xi(kr)\cos(\xi\vartheta)$, $\bx_0=(0,0.5)$ in the
  domain $\Omega=(0,1)^2$, for $k=10$ and $\xi=1$ (left), $\xi=3/2$
  (center), and $\xi=2/3$ (right).}
  \label{fig:sing}
\end{figure}

We report in Figure~\ref{fig:sing_errors} the relative errors in the
$L^2$-norm for the three cases, obtained on the structured triangular
mesh with 32 elements, for different values of $p$.
While for the smooth solution ($\xi=1$) we have exponential
convergence, before instability takes place, as in the previous
experiments, in the other two cases
the convergence is algebraic, even if the rates are 
unclear, the results for $\xi=3/2$ being better than those for $\xi=2/3$.
Similar results were obtained in \cite{PVersion} for the PWDG method.
We also report the results obtained with the PUM, for the sake of
comparison (this is the reason why we have used triangular meshes).


\begin{figure}[!htbp]
  \centering
\hspace{-1truecm}
  \begin{minipage}[c]{0.50\textwidth}
\psfrag{plogp}[c]{\scriptsize $p/\log(p)$}
\psfrag{error}[c]{\scriptsize $\log(L^2\text{-error})$}
    \qquad\qquad\includegraphics[width=0.80\textwidth]{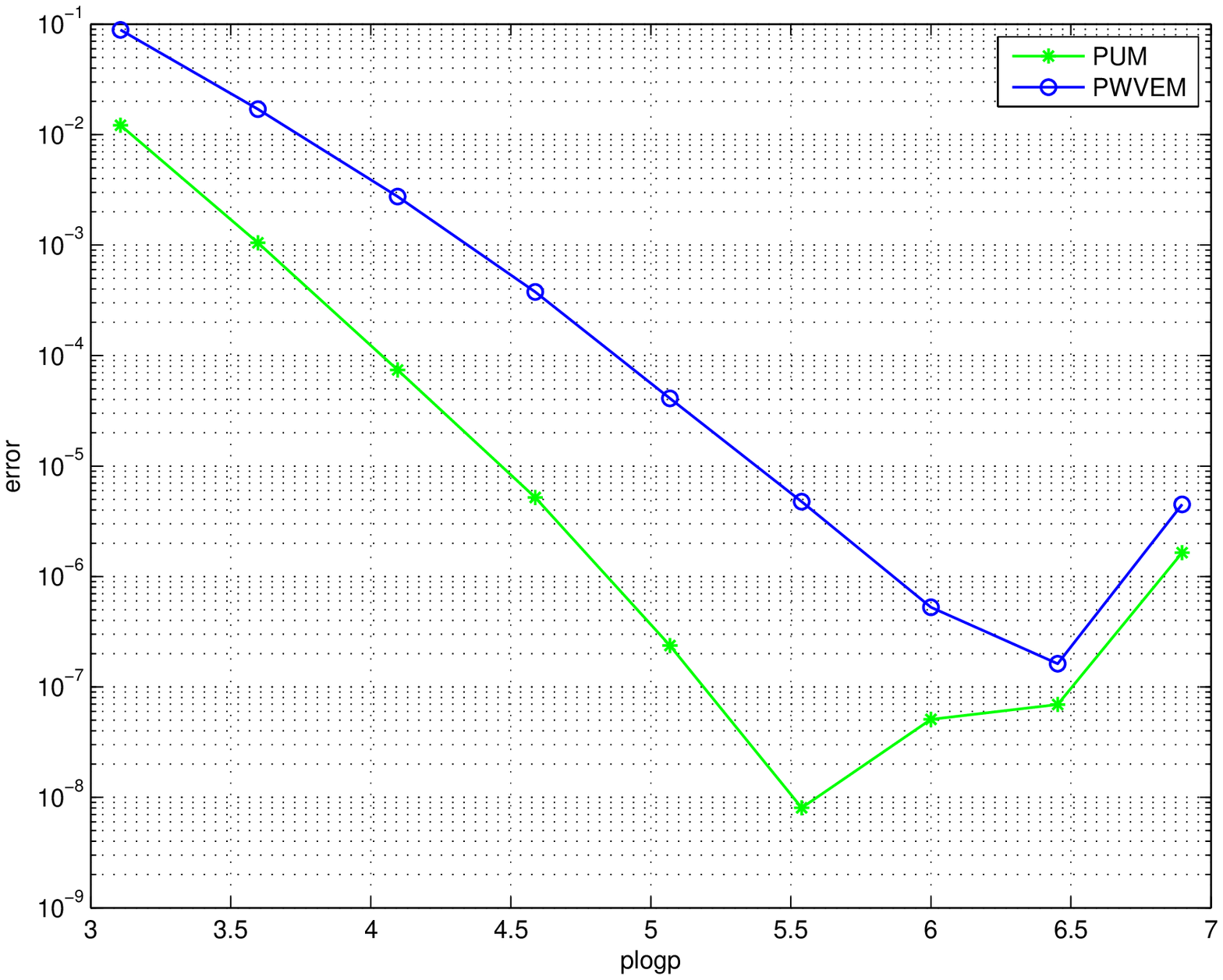}
  \end{minipage}%
  \begin{minipage}[c]{0.50\textwidth}
\psfrag{plogp}[c]{\scriptsize $p/\log(p)$}
\psfrag{error}[c]{\scriptsize $\log(L^2\text{-error})$}
\qquad\qquad\includegraphics[width=0.80\textwidth]{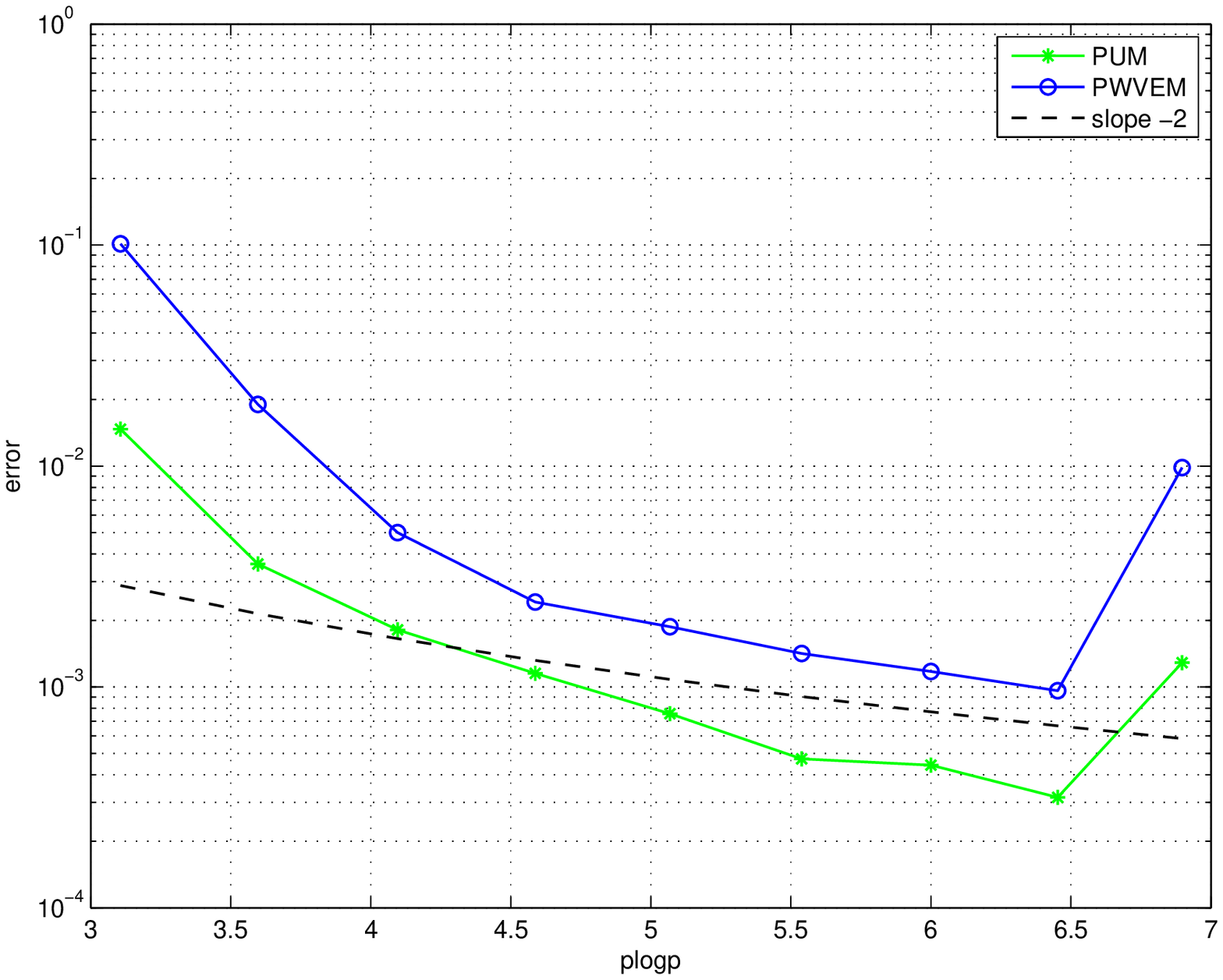}
  \end{minipage}\\
  \begin{minipage}[c]{0.50\textwidth}
\psfrag{plogp}[c]{\scriptsize $p/\log(p)$}
\psfrag{error}[c]{\scriptsize $\log(L^2\text{-error})$}
\qquad\qquad\includegraphics[width=0.80\textwidth]{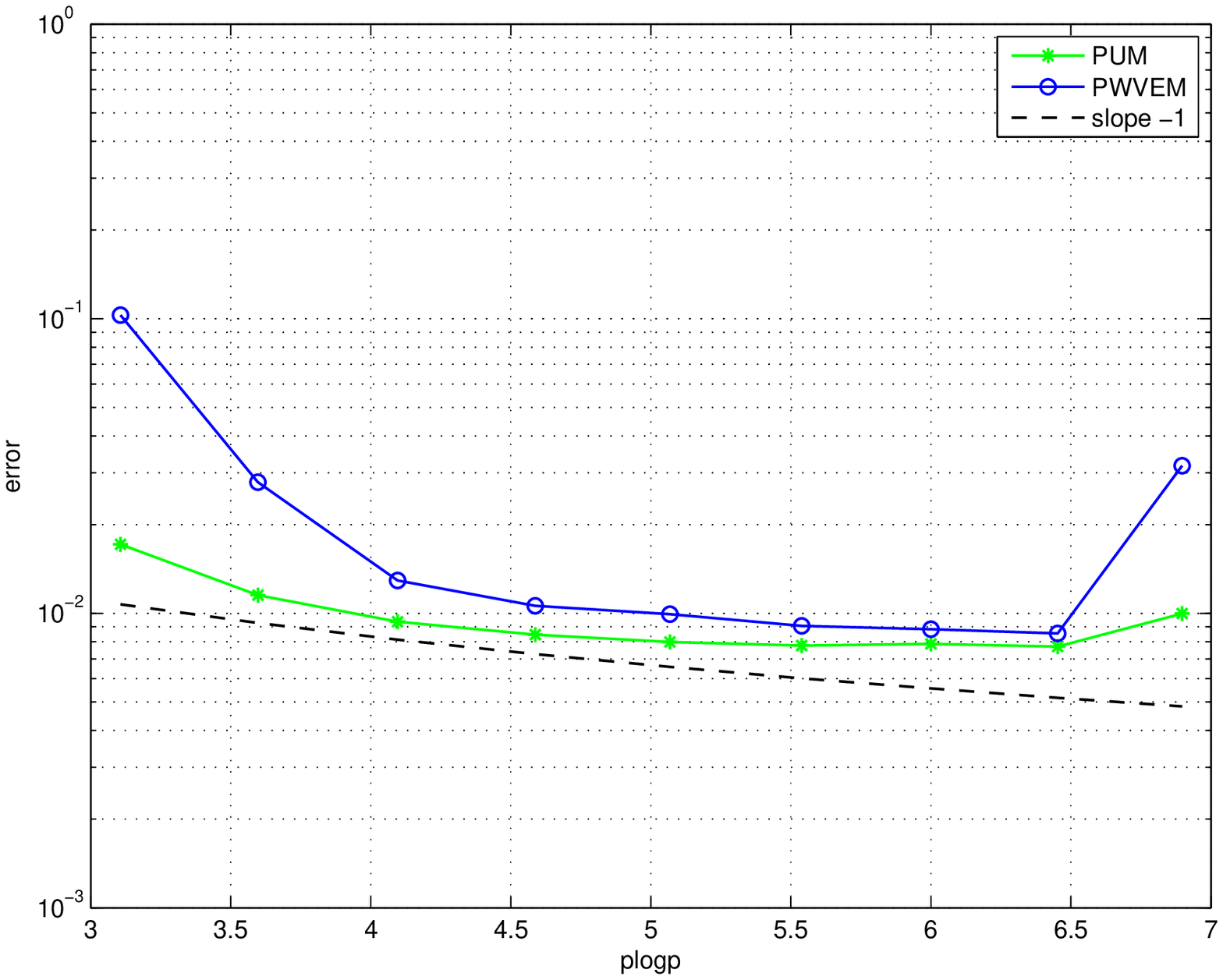}
  \end{minipage}
\caption{Relative error in the $L^2$-norm of the PW-VEM and PUM on the
  structured triangular mesh with 32 elements, for the model problem
  with exact solution~\eqref{eq:singsol} for $k=10$ and $\xi=1$ (top,
left), $\xi=3/2$ (top, right), and $\xi=2/3$ (bottom), for different
values of $p$.}
  \label{fig:sing_errors}
\end{figure}

\section{Conclusions}\label{sec:concl}
We have presented a first PW-VEM scheme for the discretization
of the Helmholtz equation. $H^1$-conformity is guaranteed by the VEM
framework, while high order convergence for the homogeneous problem, for smooth analytical
solutions, is achieved through a plane wave enrichment of the
approximating spaces. 
An $h$-version error analysis of the PW-VEM is derived, while
  numerical results also show exponential convergence of the
  $p$-version for smooth analytical solutions.
These preliminary results highlight that
a suitable interplay of $h$, $p$ and $k$ is important in order to
obtain quasi-optimality, and suggest that similar results as those of
the complete $hp$-analysis of~\cite{MS2011} could hold.

Several restrictions on the model problem we
have considered could be easily removed, and the method could be
extended from 2D to 3D, and to acoustic scattering problems. 
The extension of the method and its analysis to problems with non constant
coefficients or non zero source terms is also an interesting issue. 
Exploring alternative choices for the stabilization term, the
projection operator $\Pi$ and/or of the space onto which to project,
as well as extensions to non uniform $p$ and higher order VEM will be
subject of future research.

\section*{Acknowledgment}
The authors are grateful to L. Beir$\tilde {\rm a}$o da Veiga, F. Brezzi and L.D. Marini for stimulating and fruitful discussions.

Ilaria Perugia and Paola Pietra acknowledge support of the Italian Ministry of Education, University and Research (MIUR) through the project PRIN-2012HBLYE4.



\end{document}